\documentclass[10pt]{article}
\usepackage {amssymb} %, adjustbox, enumerate, amsbsy}
\usepackage {amsmath}
\usepackage{amsthm}
\usepackage{graphicx}
\usepackage {amscd}
\usepackage{subfigure}

\usepackage{color}
\usepackage{multirow}
\usepackage{tabu}

\usepackage{diagbox}
\usepackage{xypic}
\usepackage[colorlinks, citecolor=blue]{hyperref}

\makeatletter
\newcommand{\thickhline}{%
    \noalign {\ifnum 0=`}\fi \hrule height 1pt
    \futurelet \reserved@a \@xhline
}
\newcolumntype{t}{@{\hskip\tabcolsep\vrule width 1pt\hskip\tabcolsep}}
\makeatother

\newcommand*{\LargerCdot}{{\raisebox{0ex}{\scalebox{2}{$\cdot$}}}}

\newcommand{\vol}{{\rm vol}}
\newcommand{\ord}{{\rm ord}}
\newcommand{\fm}{\mathfrak{m}}
\newcommand{\fa}{\mathfrak{a}}
\newcommand{\cO}{\mathcal{O}}
\newcommand{\bR}{\mathbb{R}}
\newcommand{\bC}{\mathbb{C}}
\newcommand{\bZ}{\mathbb{Z}}
\newcommand{\lct}{{\rm lct}}
\newcommand{\wt}{{\rm wt}}

\newcommand{\len}{{\rm length}}

\newcommand{\Val}{{\rm Val}}

\newcommand{\hvol}{{\widehat{\rm vol}}}
\newcommand{\bx}{{\bf x}}
\newcommand{\Vol}{{\rm Vol}}

\newcommand{\hV}{{\widehat{V}}}

\newcommand{\bQ}{{\mathbb{Q}}}

\newcommand{\bP}{{\mathbb{P}}}

\newtheorem{thm}{Theorem}[section]

\newtheorem{lem}[thm]{Lemma}
\newtheorem{cor}[thm]{Corollary}
\newtheorem{defn}[thm]{Definition}

\newtheorem{prop}[thm]{Proposition}

\newtheorem{conj}[thm]{Conjecture}
\newtheorem{rem}[thm]{Remark}
\newtheorem{examp}[thm]{Example}
\begin{document}

\title{Minimizing normalized volumes of valuations}
\author{Chi Li}
%\date{}

\maketitle{}

\abstract{
For any $\bQ$-Gorenstein klt singularity $(X,o)$, we introduce a normalized volume function $\hvol$ that is defined on the space of real valuations centered at $o$ and consider the problem of minimizing $\hvol$. We prove that the normalized volume has a uniform positive lower bound by proving an Izumi type estimate for any $\bQ$-Gorenstein klt singularity. Furthermore, by proving a properness estimate, we show that the set of real valuations with uniformly bounded normalized volumes is compact, and hence reduce the existence of minimizers for the normalized volume functional 
$\hvol$ to a conjectural lower semicontinuity property. We calculate candidate minimizers in several examples to show that this is an interesting and nontrivial problem. In particular, by using an inequality of de-Fernex-Ein-Musta\c{t}\u{a}, we show that the divisorial valuation associated to the exceptional divisor of the standard blow up is a minimizer of $\hvol$ for a smooth point. Finally the relation to Fujita's work on divisorial stability is also pointed out. }

\tableofcontents

\section{Introduction and main results}

This paper is motivated by the following question:
\vskip 1mm
%\noindent
{\it
Given a $\bQ$-Gorenstein klt singularity $(X, o)$, is there an optimal way of rescaling it to obtain an affine cone that has a canonical metric structure?
}
\vskip 1mm

This will be made more precise later, and here we just give some rough explanation. A rescaling in the above question will be represented by a real valuation centered at the singularity $o\in X$ and the affine cone is given by the spectrum of the associated graded ring (assuming that the latter is finitely generated). To define the ``optimal way", we first introduce the normalized volume function that is defined 
on $\Val_{X,o}$, the space of real valuations that are centered at $o\in X$:
\begin{equation}
\hvol(v)=A(v)^n \vol(v), \text{ for any } v\in \Val_{X,o}.
\end{equation} 
Here $A_X(v)$ denotes the log-discrepancy of $v$ and $\vol(v)$ denotes the volume of $v$. We refer to Section \ref{secval} for their definitions. Then we seek for the minimizer of $\hvol(v)$ among all real valuations that are centered at $o$.
%\item The affine cone should have a possibly singular almost K\"{a}hler Ricci-flat metric cone structure. More precisely, the affine cone should degenerate to an affine variety with a singular K\"{a}hler Ricci-flat metric cone structure.
%\end{itemize}
%Actually we conjecture that the above two properties should ultimately be equivalent to each other (see part 5 of Conjecture \ref{mainconj}).
% we would like the affine cone to have a K\"{a}hler Ricci-flat structure. will first define a function, called the normalized volume function, on the space of all valuations centered at the singularity. The definition of the normalized volume function is motivated by the study in Sasaki-Einstein metric geometry. 

This minimization question is motivated by the recent study of K\"{a}hler-Einstein metrics (we will refer to section \ref{secGH} for some discussion of the study that motivates our problem). %One of the aims of the project proposed here is to provide an algebraic characterization of the metric tangent cone by showing that the first step in Donaldson-Sun's procedure in \cite{DS15} is equivalent to a procedure of minimizing normalized volumes of valuations (see part 5 of Conjecture \ref{mainconj}). 
The transition to the current purely algebro-geometric question is however inspired by a corresponding volume minimization phenomenon in Sasaki-Einstein geometry that was discovered by Martelli-Sparks-Yau in \cite{MSY08} (see also \cite{CS12}), which in some sense is a special case of the general procedure discussed here (cf. Section \ref{secSG} and \cite{LL16}). A key observation that leads to the much more general setting is that a normalization involving weights of holomorphic volume forms in Sasaki-Einstein geometry can be equivalently rephrased by using the log discrepancy of valuations (see Section \ref{secSG}).

%We point out that one advantage of our algebro-algebraic problem is that it can be studied for all $\bQ$-Gorenstein klt singularities that do not necessarily come from the Gromov-Hausdoff limits as those studied in \cite{DS15} (see Section \ref{secGH}). 
%We will show in \cite{LL16} and \cite{LX16}, for many klt singularities, the minimizers of the normalized volume functions exist as canonical objects associated to the singularities. We conjecture this to be true for any $\bQ$-Gorenstein klt singularity (see Conjecture \ref{mainconj} and its refinement in \cite{LX16}).

In this paper we shall introduce the basic set-up of this minimization problem, illustrate it by motivating examples and prove estimates that will be useful for further developments. 

%(which turns out to be the case, cf. \cite{Li15b, LL16, LX16}, see also \cite{Liu16, Blu16}). 
A basic estimate we prove is stated in the following theorem.

\begin{thm}\label{proper}
Assume $(X,o)$ is a $\bQ$-Gorenstein klt singularity.
There exists a positive constant $K=K(X,o)>0$ such that for any valuation $v$ centered at $o\in X$ with $A_X(v)<+\infty$, the following inequality holds:
\begin{equation}\label{properineq}
A_X(v)^n\cdot \vol(v)\ge K \frac{A_X(v)}{v(\fm)},
\end{equation}
where $\fm$ is the maximal ideal defining $o\in X$.
\end{thm}
The infimum of the left-hand-side in \eqref{izuvol} is a new invariant of the klt singularity %(compare \cite{BFF12, Zha14}) 
and seems interesting to be studied further. 
To prove this estimate, there are two main ingredients that are of independent interests. The first is an Izumi-type estimate which generalizes a well-known estimate for smooth points:
%As a starting estimate, we prove:
%\begin{thm}\label{thm-UPLB}
%Let $(X,o)$ be a $\bQ$-Gorenstein klt singularity. Then 
%there exists a constant $C=C(X,o)>0$ such that for any $v\in \Val_{X,o}$ we have:
%\begin{equation}\label{izuvol}
%A_X(v)^{n}\cdot \vol(v)\ge %\frac{1}{c^n}s(v)^{n}\cdot \vol(v)\ge 
%C>0.
%\end{equation}
%where $e(\fm)$ is the Hilbert-Samuel multiplicity of $R$ along $\fm$.
%\end{thm}
%For later purposes we introduce:
%\begin{defn}\label{Def-volX}
%For any $\bQ$-Gorenstein klt singularity $(X, o)$, we define its normalized volume to be the following positive number:
%\[
%\hvol(X,o):=\inf_{v\in \Val_{X,o}}\hvol(v).
%\]
%\end{defn}
%The uniform positive lower bound in \ref{thm-UPLB} follows from the following Izumi-type estimate that generalizes the corresponding well-known estimate in the smooth point case.
\begin{thm}[=Theorem \ref{Izumi}]\label{Thm-Izumi}
Let $(X,o)$ be a $\bQ$-Gorenstein klt singularity.
Then there exists a constant $c_1=c_1(X,o)>0$ such that
\begin{equation}\label{kltizu}
v(\fm)\ord_o\le v\le c_1 \cdot A_X(v)\ord_o,
\end{equation}
for any valuation $v$ centered at $0$. As a consequence, there is a uniform lower bound $\hvol(v)\ge \mathfrak{e}(\mathfrak{m})/c_1^n>0$ for any $v\in \Val_{X,o}$, where $e(\fm)$ is the Hilbert-Samuel multiplicity of $\fm$.
\end{thm}
The other ingredient is the following more technical estimate, which is related to a volume formula of Favre-Jonsson for $v\in \Val_{\bC^2, 0}$.
\begin{thm}[=Theorem \ref{mainthm}]
Assume $(X,o)$ is a $\bQ$-Gorenstein klt singularity.
There exists a positive constant $c_2=c_2(X,o)>0$ such that for any valuation centered at $o$, the following inequality holds:
\begin{equation}\label{keyineq}
\vol(v)\ge c_2 \left(\sup_{\fm}\frac{v}{\ord_o}\right)^{1-n}\frac{1}{v(\fm)}.
\end{equation}
\end{thm}

%Combining the above two results, we obtain a properness estimate:
%\begin{rem}
%To the author, the estimate \eqref{properineq} is heuristically a valuative version of Moser-Trudinger inequality that arised in the study of K\"{a}hler-Einstein metric problem. Considering the close relation between the valuation space and Berkovich space that has been explored deeply, it seems interesting to interpret the above result using the language of non-Archimedean geometry. In this regard, we observe that the right-hand-side of \eqref{properineq} corresponds to the weight function studied in \cite{MN12} in the non-Archimedean setting.
%\end{rem}
The estimate in Theorem \ref{proper} strongly suggests that the minimizer should exist. In fact, Theorem \ref{proper} reduces the existence to the lower semicontinuity of $\hvol$ on $\Val_{X,o}$ (see Corollary \ref{maincor}). We make the following
\begin{conj}\label{conj-existence}
For any $\mathbb{Q}$-Gorenstein klt singularity $(X, o)$, there exists a {\it unique} minimizing valuation of $\hvol$ on $\Val_{X,o}$.
\end{conj}
We point out that the Conjecture \ref{conj-existence} is non-trivial even for the smooth case. By using an inequality of de-Fernex-Ein-Musta\c{t}\u{a}, we will prove in Proposition \ref{prop-sm} that the exceptional divisor of the standard blow up is a minimizer. In a following paper \cite{Li15b}, we will study the case when $X$ is an affine cone over a K-semistable $\mathbb{Q}$-Fano variety.

Finally by deriving the volume formula for specific divisorial valuations, we point out that there is a close relation of our minimization problem to the work of Fujita on divisorial semi-stability. 

We end this introduction by outlining the organization of the paper. In the next section, we recall the definition of volumes and log discrepancies for valuations and explain a key observation on Sasaki-Einstein geometry that leads to our formulation of the problem. In section \ref{SecIzumi}, we prove the Izumi type estimate stated in Theorem \ref{Thm-Izumi}. In section \ref{sec-minhvol}, we prove Theorem \ref{proper} by proving the technical estimate in Theorem \ref{mainthm}. In Section \ref{SecExmp}, we discuss several examples including the case of smooth point. In section \ref{secFujita}, we point out the relation between the minimizations of normalized volumes and Fujita's work. In section \ref{sec-conj} we propose several conjectures that are natural from our point of view. In Appendix I, we write down the candidate minimizer for $D$-type and $E$-type singularities. In Appendix II, following the suggestion of a referee, we give an alternative proof of Theorem \ref{thm-UIzumi} by using the argument of Boucksom-Favre-Jonsson.

%Interestingly as it gradually turns out, this minimization problem is indeed related to several subjects in complex differential/algebraic geometry, including K\"{a}hler-Einstein/Sasaki-Einstein metrics (see \cite{LL16}), (log-)K-semistability (see Section \ref{secFujita} and \cite{Li15b, LL16}), (sharp) deFernex-Ein-Musta\c{t}\u{a} type inequalities on klt singularities (see \cite{Liu16}, \cite{LX16}) and plt blow-ups (see \cite{LX16}).

\section{Preliminaries}\label{secpre}

\subsection{Volumes and log discrepancies of valuations}\label{secval}

Let $X^n$ be an $n$-dimensional normal affine variety. Fix a closed point $o\in X$ and let $R:=\mathcal{O}_{X,o}$ be the local ring of $X$ at $o$ with the maximal ideal denoted by $\fm$.  Let $\Val_{X,o}$ denote the space of all real valuations on $\bC(X)$ with center $o$ on $X$. For any $v\in \Val_{X,o}$, we define the volume of $v$ following Ein-Lazarsfeld-Smith (see \cite{ELS03}):
\begin{equation}\label{defvol}
{\rm vol}(v)=\limsup_{r\rightarrow +\infty}\frac{{\rm length}_R(R/\fa_r)}{r^n/n!},
\end{equation}
where $\fa_r=\{f\in R; v(f)\ge r\}$.
By \cite{ELS03, Mus02, Cut12}, we know that:
\begin{equation}\label{vol=mult}
\vol(v)=e(\fa_{\LargerCdot}):=\lim_{r\rightarrow +\infty}\frac{e(\fa_r)}{r^n}.
\end{equation} 
%Here $e(\fa_{\LargerCdot})$ is defined as the normalized  limit of Hilbert-Samuel multiplicities
Here the Hilbert-Samuel multiplicity $e(\fa_r)$ is defined as follows:
\[
e(\fa_r)=\lim_{p\rightarrow +\infty}\frac{\len(R/\fa_r^p)}{p^n/n!}.
\]

%\[\Bigcdot_{i=1}^\infty A_i=A_1+A_2+\cdots\]

From now on, we will restrict our attention to the class of $\mathbb{Q}$-Gorenstein klt singularities. Recall the following standard definitions:
\begin{defn}
\begin{enumerate}
\item
$(X, o)$ is $\mathbb{Q}$-Gorenstein if there exists an integer $m\in \mathbb{Z}$ such that $mK_X$ is a Cartier divisor near $o\in X$.
\item
Let $(X,o)$ be a $\mathbb{Q}$-Gorenstein singularity.
$(X, o)$ is klt (Kawamata log terminal) if for any divisor $E$ over $(X,o)$, we have $A_X(E)>0$.
\end{enumerate}
\end{defn}

Following \cite{JM10} and \cite{BFFU13}, we briefly recall the definition of log discrepancy function $A(v)=A_X(v)$
for any valuation centered at $o$: $v\in \Val_{X,o}$. If $v=\ord_E$ is a divisorial valuation for
an exceptional divisor $E$ over $o$ such that there is a birational morphism $\pi: Y\rightarrow X$ and $E$ is a prime divisor on $Y$ that is contracted to $o\in X$, then 
we define its log discrepancy as the coefficient of $E$ in $K_{Y/X}+E$:
\[
K_Y+E=\pi^*K_X+A(E)E+F,
\]
where ${\rm Supp}(F)$ does not contain $E$.
In general, one first defines the log discrepancy for any quasi-monomial valuation as follows. Assume $\mu: Y\rightarrow X$ is a resolution of singularities. $X$ and $\underline{y}=(y_1,\dots, y_r)$ is a system of algebraic coordinates at a point $y\in Y$. By \cite[Proposition 3.1]{JM10}, to every $\alpha=(\alpha_1, \dots, \alpha_r)\in \mathbf{R}^{n}_{\ge 0}$ one can associate a unique valuation ${\rm val}_\alpha={\rm val}_{\underline{y},\alpha}\in \Val_X$ with the following property: whenever $f\in \mathcal{O}_{Y,\eta}$ is written in $\widehat{\mathcal{O}_{Y,\eta}}$ as $f=\sum_{\beta\in\mathbb{Z}^{r}_{\ge 0}}c_\beta y^{\beta}$, with each $c_\beta$ either zero or unit, we have
\[
{\rm val}_\alpha(f)=\min\{\langle\alpha,\beta\rangle | c_\beta\neq 0\}.
\]
The set of all such valuations (called quasi-monomial valuations or equivalently Abhyankar valuations, see \cite{ELS03}) is denoted by ${\rm QM}_\eta(Y,D)$. If $\eta$ is the generic point of a connected component of intersection of $D_1, \dots, D_r$ of $D$,  then the map
 \begin{equation}\label{quasisom}
 {\rm QM}_{\eta}(Y,D)\rightarrow\bR^r, \quad v\mapsto (v(D_1), \dots, v(D_r))
\end{equation}
 gives a homeomorphism onto the cone $\bR^{r}_{\ge 0}$ (see \cite[Lemma 4.5]{JM10}).
The log discrepancy function for such quasi-monomial valuation is defined as follows:
\begin{equation}\label{ldquasi}
A_X(v)=\sum_{i=1}^r v(D_i)\cdot A_X(\ord_{D_i})=\sum_{i=1}^r v(D_i)\cdot (1+\ord_{D_i}(K_{Y/X})).
\end{equation}
In \cite{JM10}, Jonsson-Musta\c{t}\u{a} defined $A(v)$ for any real valuation on $\bC^n$. They first showed that there is a retraction map $r: \Val_X\rightarrow {\rm QM}(Y,D)$ such that it induces a homomeomorphism
\[
r: \Val_X\rightarrow \lim_{\stackrel{\longleftarrow}{(Y,D)}}{\rm QM}(Y,D).
\]
For any real valuation $v\in \Val_{X,o}$, one can then define:
\[
A(v)=A_X(v):=\sup_{(Y,D)} A(r_{Y,D}(v)),
\]
where the supremum is over all log-smooth pairs $(Y,D)$ over $X$. Jonsson-Musta\c{t}\u{a}'s construction has been generalized to the singular case in \cite{BFFU13} by following the same scheme of approximations.

\begin{rem}
In this paper we will only work with the class of $\mathbb{Q}$-Gorenstein singularities. Notice that the log discrepancy and klt condition can be defined for all normal singularities following the work of de-Fernex-Hacon \cite{dFH09} (see also \cite{BFF12, BFFU13}) and most discussions in this paper can indeed be generalized correspondingly (pointed out to me by a referee). However to avoid technical complications we leave the generalization to future works.
\end{rem}

Notice that with the above definitions and notations,  $A(\lambda v)=\lambda A(v)$ and ${\rm vol}(\lambda v)=\lambda^{-n}{\rm vol}(v)$ for any $\lambda>0$ and $v\in \Val_{X,o}$. So the function
$A_X(v)^n\vol(v)$ is a scaling invariant function on $\Val_{X,o}$. This function is well defined if $A_X(v)<+\infty$ (see Section \ref{sec-minhvol}). 
Where there is a torus action on $(X, o)$ with $o$ being the unique attracting point, the restriction of $\hvol$ to the space of toric invariant valuations appeared in Sasaki-Einstein geometry (in a disguised form) which we will discuss next.

\subsection{Weights vs. log discrepancies}\label{secSG}
In this section, we relate the weight of holomorphic volume form to the log discrepancy of torus invariant valuation. This is the key observation that leads to our general minimization problem. 

Assume that $X={\rm Spec}(\bC[z_1,\dots, z_n]/I(X))$ is a normal affine variety over $\bC$ with a $(\mathbb{C}^*)^{r}$-action. Assume that the action is free outside $o$ and $o$ is an attracting fixed point. We have the weight decomposition:
\[
R:=\bC[z_1, \dots, z_n]/I(X)=\bigoplus_{\alpha\in \Gamma^*}R_{\alpha}.
\]
We assume that $R_\alpha\neq 0$ for any $\alpha\in \Gamma^*$.
The Reeb cone is a cone $\mathcal{C}\subset \mathfrak{t}$ consisting of all $\xi\in \mathfrak{t}\cong \bR^{r}$ such that $\alpha(\xi)>0$ for all $\alpha\in \Gamma^*$. If $\xi\in \mathcal{C}$, then one can define the index character:
\[
f(t,\xi)=\sum_{\alpha\in \Gamma^*}e^{-t\alpha(\xi)}\dim R_\alpha,
\]
which has an expansion (\cite{MSY08}, \cite{CS12}) near $t=0$:
\[
f(t,\xi)=\frac{a_0(\xi)}{t^n}+O(t^{1-n}).
\]
It's easy to see that $a_0(\xi)$ is a homogeneous function of $\xi$ of homogeneous degree of $-n$ (see \cite{MSY08}).

If $\xi$ is rational, then $e^{t\xi}$ generates a $\mathbb{C}^*$ action. $(E,\Delta)=X/\langle e^{t\xi}\rangle$ is an algebraic stack and there is an orbifold line bundle $L\rightarrow (E, \Delta)$ such that the underlying variety $X$ becomes an orbifold cone over $(E, \Delta)$. In other words we have:
\[
E={\rm Proj}\bigoplus_{m=0}^{\infty} H^0(E, mL)
\mbox{ and } X={\rm Spec}\bigoplus_{m=0}^{\infty}H^0(E, mL):=C(E,L).
\]
If $\xi=w\partial_w$ is the standard rescaling vector where $w$ is a linear coordinate along the fibre of $X\setminus\{o\} \rightarrow E$, then it's easy to verify that
$
a_0(\xi)=L^{n-1}.
$
On the other hand, any Reeb vector field $\xi$ determines a real valuation $\wt_\xi\in \Val_{X,o}$ in the following way:
\[
{\rm wt}_\xi(f)=\min_{\alpha\in \Gamma^*}\{\langle \alpha, \xi\rangle; f_{\alpha}\neq 0\} \text{ for any } f=\sum_\alpha f_\alpha\in R.
\]
\begin{lem}
$\vol({\rm wt}_\xi)=a_0(\xi)$ for $\xi=w\partial_w$.
\end{lem}
\begin{proof}
Since both sides are homogeneous of degree $-n$ in $\xi$. We can normalize $\xi$ such that it corresponds to the standard scaling along
the fibre of $L$. Then it's easy to see that
\[
\fa_r({\rm wt}_\xi)/\fa_{r+1}({\rm wt}_\xi)\cong H^0(X, L^r).
\]
So we can calculate:
\begin{eqnarray*}
\dim_{\bC}R/\fa_r(\wt_\xi)&=&\sum_{m=0}^{r}d_m=\sum_{m=0}^r \left(a_0(\xi)\frac{m^{n-1}}{(n-1)!}+O(m^{n-2})\right)\\
&=&a_0(\xi)\frac{r^{n}}{n!}+O(r^{n-1}).
\end{eqnarray*}
By definition of volume in \eqref{defvol}, we get the identity.
\end{proof}
From now on, we assume that $ -(K_E+\Delta)\sim_{\mathbb{Q}} r L$ for $r>0$. We relate the weights on the holomorphic $n$-form to the log-discrepancies. 
Let $\pi: Y\rightarrow X$ be the extraction (blow-up) of $E$. Then we can write:
\begin{equation}\label{bld1}
K_Y+E=\pi^*(K_X)+A_X(E)E.
\end{equation}
$\ord_E$ is a valuation on $R$ such that: $\ord_E(f)=m$ if $f\in H^0(E, -m(K_E+\Delta))$. 

In the simplest case, $X=C(E, -K_E)$. $X$ has a nonzero
vanishing $n$-form given by $\Omega=dz\wedge dw$
where $w$ is the fibre variable $w dz$. The holomorphic vector field $\xi=w\partial_w$ has the property that:
\[
\mathcal{L}_\xi\Omega=\Omega. %,\quad \mathcal{L}_\xi f=\ord_E (f) f.
\]
On the other hand, if we restrict \eqref{bld1} on $E$ we get: $K_E=A_X(E)E|_E \Rightarrow A_X(E)=1$. 
%If $-K_E=r (-E)|_E$ with $\lambda>1$, then the restriction gives $K_E=A_X(E)E|_E \Rightarrow A_X(E)=r$. The nonzero
%vanishing $n$-form is (at $w (dz)^{1/r}$)
%\[
%\Omega=d w^{r}\wedge dz=r w^{r-1} dw\wedge dz.
%\]
%$V=w\partial_w$ has weight $\mathcal{L}_V\Omega=r\Omega$.
More generally, assume that $(E,\Delta)$ is a log-Fano variety such that $-(K_E+\Delta)\sim_{\mathbb{Q}} r\cdot L$. Then the cone $C(E,L)$ is a Calabi-Yau variety with nonvanishing $n$-form $\Omega$. At $w\cdot s$, we have a non-zero holomorphic volume form:
\[
\Omega= dw^{r}\wedge s^{\otimes r}=dw^r\wedge \mu_X^*dz,
\]
where $\mu_X: Y\rightarrow E$ is the projection. The standard vector $\xi=w\partial_w$ satisfies $\mathcal{L}_\xi\Omega=r\Omega$. 
On the other hand, by adjunction, we have: $K_E+\Delta_E=A_X(E) E|_E=A_X(E) (-L)|_E$. So we get $A_X(\ord_E)=A_X(E)=r$ (see \cite[Section 3.1]{Kol13}). Notice that the valuation $\ord_E$ coincides with ${\rm wt}_\xi$. 

\begin{rem}
The above discussion can be understood in the setting of the well-known Dolgachev-Pinkham-Demazure construction. In the above we just gave simplified statements to motivate our problem. 
\end{rem}

So we see that the weight of the action indeed corresponds to log discrepancy. This is also well known in the toric case. Consider a toric affine variety $X_\sigma$ defined by a polyhedral cone $\sigma\subset \mathbb{R}^n=N\otimes_{\mathbb{Z}}\mathbb{R}$ generated by
primitive vectors $v_1,\dots, v_n\in N_{\mathbb{Z}}$. Let $\sigma^{\vee}\subset M_\mathbb{R}=\left(\mathbb{R}^n\right)^{\vee}$ be the dual cone. Assume that there is a rational $\gamma\in \sigma^{\vee}\cap M_{\mathbb{Q}}\subset\left(\mathbb{R}^n\right)^{\vee}$ such that
\[
\langle \gamma, v_i\rangle=1.
\]
This is equivalent to the condition that $X_\sigma$ is $\mathbb{Q}$-Gorenstein.
For any rational point $x\in \sigma\cap N_\mathbb{Q}$, $x$ determines a valuation $v_{x}$
such that $v_x(f_y)=\langle y, x\rangle $ for any $y\in M_{\mathbb{Z}}$. On the other hand, the
toric blow up $\pi_x: Y_x\rightarrow X$ determined by $x$ is given by the division of $\sigma$ into
subcones. There is a unique exceptional divisor $E_x$ such that the log discrepancy is given by (see \cite{Bor97})
\[
A(E_x; X)=\langle \gamma, x\rangle,
\]
which is nothing but the weight of the $\Omega$ with respect to $x$ (see \cite[(49)]{FOW09}).

In the study of Sasaki-Einstein metrics, it was Martelli-Sparks-Yau who realized that the problem of minimizing $a_0(\xi)$ under the constraint $\mathcal{L}_{\xi}\Omega= n\Omega$ is important because the minimizer is the only candidate of Reeb vector field for which there could exist a Sasaki-Einstein metric (see \cite{MSY08}, \cite{GMSY07}). From the above discussion, we see that the problem is the same as minimizing the function $\hvol({\rm wt}_\xi)=A_X(\wt_\xi)^n\vol(\wt_\xi)$ among $\xi\in\mathcal{C}\subset \mathfrak{t}$. 

%\section{Normalized volume function}

\section{Izumi type estimate}\label{SecIzumi}

The following inequality will be crucial for us. If $(X,o)$ is smooth, it is well known as shown in \cite[Theorem 2.6]{ELS01} and \cite[Proposition 5.1]{JM10}. 
%We provide a proof here  for the convenience of the reader. 
\begin{thm}[Izumi type estimate]\label{Izumi}
Assume $(X,o)$ is a $\bQ$-Gorenstein klt singularity.
There exists a constant $c_1=c_1(X,o)>0$ such that
\begin{equation}\label{kltizu}
v(\fm)\ord_o\le v\le c_1 \cdot A_X(v)\ord_o,
\end{equation}
for any valuation centered at $0$.
%As a consequence, we have:
%\[
%A_X(v)^n\vol(v)\ge c\cdot e(\fm)>0.
%\]
\end{thm}
\begin{proof}
To get the lower bound, we note that $f\in \fm^{\ord_o(f)}-\fm^{\ord_o(f)+1}$. So
\[
\ord_o(f) v(\fm)\le v(f).
\]
For the upper bound, by the way of definition of $A_X(v)$ recalled in \ref{secval} (see \cite{JM10} and \cite{BFFU13}), we can assume that $v$ is divisorial valuation with 
center at $o$ by approximating the general real valuation by divisorial valuations. 

%In the following, we will give two different proofs of upper bounds. 
%\begin{enumerate}
%\item (First proof of the upper bound)
The proof of the upper bound is inspired by \cite[Proof of Theorem 8.13]{BHJ15}.
When $X$ is smooth, it follows from Skoda's criterion for integrability
\begin{equation}\label{sko2izu}
\frac{A_X(v)}{v(f)}\ge \lct(f)=\min_{v}\frac{A_X(v)}{v(f)}\ge \frac{1}{\ord_o(f)}
\end{equation}
for any $f\in \mathcal{O}_{\mathbf{C}^n,0}$ (see \cite[Proposition 5.10]{JM10}). So we have
\[
v(\fm)\ord_o\le v\le A_X(v)\ord_o.
\]
%and
%\[
%\vol(v)\ge \vol(A_X(v)\ord_o)=A_X(v)^{-n}\vol(\ord_o).
%\]
When $X$ is only klt and $\bQ$-Gorenstein, we choose a log resolution $\mu: X'\rightarrow X$ and write $\mu^*K_X=K_{X'}+B'$ so that: 
\[
A_X(v)=A_{(X',B')}(v)=A_{X'}(v)-v(B').
\]
Because $X$ is klt, we have $B'\le (1-\epsilon) B'_{\rm red}$ for some $\epsilon=\epsilon(X')>0$. Because $B'_{\rm red}$ has simple normal crossings, we have $(X', B'_{\rm red})$ is lc, i.e. 
$0\le A_{(X',B'_{\rm red})}(v)=A_{X'}(v)-v(B'_{\rm red})$. Combining above, we get 
\begin{eqnarray*}
A_X(v)= A_{X'}(v)-v(B')&=&\epsilon A_{X'}(v)+(1-\epsilon)(A_{X'}(v)-v(B'_{\rm red}))+(1-\epsilon)v(B'_{\rm red})-v(B')\\
&\ge&  \epsilon A_{X'}(v)
\end{eqnarray*}
Assume $\xi$ is the center of $v$ on $X'$. Then by the smooth case (see \cite[Proposition 5.1]{JM10}), we have:
\[
v(f)=v(\mu^*f)\le A_{X'}(v) \ord_{\xi}(\mu^*f)\le \epsilon^{-1} A_X(v)\ord_{\xi}(\mu^*f). 
\]
By Izumi's linear complementary inequality under morphisms recalled in Theorem \ref{thm-UIzumi}, we have $\ord_{\xi}(\mu^*f)\le a_2\cdot \ord_{o}(f)$ for some $a_2=a_2(X,o)\ge 1$. 
So we get the wanted estimate with $c_1=\epsilon^{-1} a_2$:
\[
v(\fm)\ord_o\le v\le \epsilon^{-1}a_2\cdot  A_X(v)\ord_o. %\le A_X(v)w,
\]
%where $w$ is any fixed Rees valuation associated to $\ord_o$.

%\item (Second proof of the upper bound) We will directly prove a Skoda type estimate, i.e. there exists a constant $c=c(X, o)>0$ such that:

\end{proof}

In the above proof, we used the following uniform version of Izmui's linear complementary inequalities. In other words, we claimed that $a_2$ in the above inequalities can be chosen to be
independent of $o'\in \mu^{-1}(o)$. Although this should hold in much
more generality (see \cite[6]{Izu07}), we just need the following version that follows from Izumi's proof.
Notice that if $\xi$ is the center of $v$ on $X'$, then $\xi$ is an irreducible subvariety of $X'$ and $\ord_{\xi}=\inf_{o'\in \xi}\ord_o'$. 
\begin{thm}[cf. \cite{Izu82, Izu85}]\label{thm-UIzumi}
Let $X$ be a normal affine variety, and $\mu: X'\rightarrow X$ be a birational morphism such that $X'$ is smooth. Assume that $Z=\mu^{-1}(o)$ is a compact subvariety of $X'$. Then there exists a constant
$a=a(X,o)$ such that for any $o'\in Z$ and any $f\in \cO_X$ we have
\begin{equation}\label{eq-2ord}
\ord_{o'}(\mu\circ f)\le a\cdot \ord_o(f).
\end{equation}
\end{thm}
\begin{proof}
We first recall how Izumi obtained \eqref{eq-2ord} for a fixed $o'\in \mu^{-1}(o)$. 
Izumi showed in \cite{Izu85} that there exist $a_1=a_1(X,o)\ge 1$ and $b_1=b_1(X,o)\ge 0$ such that
\[
\noindent
{\rm (CI_1):} \quad  \ord_o(fg)\le a_1 (\ord_o(f)+\ord_o(g))+b_1, \text{ for any } f, g\in \mathcal{O}_{X,o}.
\]
On the other hand, previously in \cite[Theorem 1.2]{Izu82} it was shown that ${\rm (CI_1)}$ implies ${\rm (CI_2)}$: there exists $a_2=a_2(o', a_1, b_1)\ge 1$ such that
\begin{equation*}%\label{eq-CI2}
\hspace{-2cm} \!\!{\rm (CI_2):} \quad \ord_{o'}(\mu\circ f)\le a_2\cdot \ord_o(f) \text{ for any } f\in \mathcal{O}_{X,o}.
\end{equation*}
To prove that $({\rm CI_1})$ implies $({\rm CI_2})$, Izumi first chose a finite surjective morphism
$\Pi: (X, o)\rightarrow (\bC^n, 0)$ and reduced the proof of ${\rm (CI_2)}$ to the following inequality (in \cite[Lemma 1.1]{Izu82}), which is the version of the inequality ${\rm (CI_2)}$ in the case where the map ($\Pi\circ \mu$) has both smooth source and target:
\begin{equation}\label{eq-Tou}
\ord_{o'}(g\circ \Pi\circ \mu)\le c\cdot \ord_0(g) \text{ for any } g\in \mathcal{O}_{\bC^n,0}.
\end{equation}
In other words, Izumi showed (in \cite[Proof of Theorem 1.2]{Izu82}) that $a_2=a_2(a_1, b_1, c)$ for $a_1, b_1$ in ${\rm CI_1}$ and $c$ in \eqref{eq-Tou}. So we just need to show that $c$, which a priorly depends on $o'\in \mu^{-1}(o)$, can be chosen to be uniform
with respect to $o'\in \mu^{-1}(o)\subset (\Pi\circ\mu)^{-1}(0)$. Since $X'$ and $\bC^n$ are smooth, the inequality \eqref{eq-Tou} was already proved in \cite[Lemma 5.6]{Tou80}. We will follow Tougeron's proof to show that the constant $c$ in \eqref{eq-Tou} is indeed uniform.  

Choose local coordinates $\{z_j\}_{j=1}^n$ around $o'\in X'$ and let $\{w_i\}_{i=1}^n$ be the flat coordinates on $\bC^n$. Then near $o'\in X'$, $\mu$ is locally given by $n$-tuples of holomorphic functions
\[
w_i=w_i(z_1, \dots, z_n), \quad 1\le i\le n.
\]
Denote the Jacobian matrix by
\[
J=(J_{ij})=\left(\frac{\partial w_i}{\partial z_j}\right).
\]
Because the map $\Pi\circ \mu$ is surjective above a neighborhood of $0\in\bC^n$, the Jacobian determinant $\det(J)=\det(\partial w_i/\partial z_j)$ is not identically equal to 0 near $o'\in X'$. Denote
$r=r(o')=\ord_{o'}\det(J)$. For any $g\in \mathcal{O}_{\mathbb{C},0}$, consider the following system of equations obtained by the chain rule:
\begin{equation}
\sum_{i=1}^n \left(\frac{\partial g}{\partial w_i}\circ \mu\right)\frac{\partial w_i}{\partial z_j}=\frac{\partial}{\partial z_j}(g\circ \mu), \quad 1\le j\le n.
\end{equation}
If $g\circ\mu \in \fm_{o'}^{k+r+1}$, then $\frac{\partial}{\partial z_j}(g\circ\mu)\in \fm_{o'}^{k+r}$ for any $1\le j\le n$. This implies 
\[
\frac{\partial g}{\partial w_i}\circ \mu= \det(J)^{-1} ({\rm adj}(J))_{ij} \frac{\partial }{\partial z_j}(g\circ\mu) \in \fm_{o'}^{k},
\]
for any $1\le i\le n$ where ${\rm adj}(J)$ is the adjugate matrix of $J$.
Iterating this argument, we get that for any multiple index $\alpha\in \mathbb{N}^n$:
\[
\frac{\partial^\alpha g}{\partial w^\alpha}\circ \mu\in \mathfrak{m}_{o'}^{k+r+1-(r+1)|\alpha|}.
\]
On the other hand, if $\ord_0 g=p$, then there is some multiple index $\beta\in \mathbb{N}^n$ with $|\beta|=p$ such that $\frac{\partial^\beta g}{\partial w^\beta}(0)\neq 0$ and hence $\frac{\partial^\beta g}{\partial w^\beta} \circ \mu\not\in \mathfrak{m}_{o'}$ . So we get the inequality $k+r+1-(r+1)p\le 0$, which implies:
\[
\ord_{o'} \left(g\circ\mu\right) \le (r(o')+1)\cdot \ord_{0} g.
\]
Now we can choose $c=\max\left\{ \ord_{o'}\det(J); o'\in (\Pi\circ\mu)^{-1}(0)\right\}$. The maximum can be obtained, because
$o'\rightarrow \ord_{o'}\det(J)$ is upper semicontinuous and $(\Pi\circ\mu)^{-1}(o)$ is a compact set.

\end{proof}

%Indeed, by the proof of Theorem 1.2 in \cite{Izu82}, we see that the constant $a_2$ depends on the constant $a$ in \cite[Lemma 1.1]{Izu82} and
%the constants $a_1=a_1(X,o)\ge 1$ and $b_1=b_1(X,o)\ge 0$ in \cite[Theorem 1.2]{Izu82}, because a fixed finite morphism $\Pi: (X, o)\rightarrow (\bC^n, 0)$ works for the estimates for any $o'\in \mu^{-1}(o)$. On the other hand, the constant $a$ in 
%\cite[Lemma 1.1]{Izu82} only depends on the morphism $\mu\circ\Pi: X'\rightarrow \bC^n$, and can be chosen to be uniform with respect to any $o'\in (\mu\circ\Pi)^{-1}(0)=\mu^{-1}(o)$ by 
%the proof of \cite[Lemma 5.6]{Tou80}.

\begin{rem}
In Appendix II, we will give a second proof of the needed uniform Izumi estimate in Theorem \ref{thm-UIzumi} following a referee's suggestion. \\
Actually there is also a proof of Theorem \ref{Izumi} using degeneration argument which involves some deep results from Minimal Model Program (MMP). This method actually allows us to prove a Skoda type estimate, i.e. there exists a constant $c=c(X, o)>0$ such that:
\begin{equation}\label{eqskoda}
\lct_o(f)\ge \frac{c}{\ord_0(f)} \text{ for any } f\in \mathcal{O}_{X,o}.
\end{equation}
Assuming \eqref{eqskoda}, we can easily complete the proof of Theorem \ref{Izumi} as in \eqref{sko2izu}.
Roughly speaking, the proof of \eqref{eqskoda} consists of two steps.
In the first step one can prove the estimate in the case that $(X, o)$ is an orbifold cone over a log Fano-pair and $f$ is a homogenous function, with the help of uniform estimates of $\alpha$-invariant by Tian and Boucksom-Hisamoto-Jonsson. In the second step one can use the lower semicontinuity of log canonical threshold to reduce the problem to the previous case by considering an equivariant degeneration of any klt singularity to an orbifold cone. It is the degeneration part that is achieved by the MMP through the notion of plt blow-ups. Following a referee's suggestion, to avoid the unnecessary complication involving MMP, we refer the interested reader to the proof presented in an earlier version of this paper on arXiv.
\end{rem}

%\begin{rem}
%We also expect that there is a proof of Proposition \ref{Izumi} along the lines of the proof in the smooth case as in \cite[Theorem 2.6]{ELS01}. However, because the subadditivity property of multiplier ideal sheaf, the key ingredient in the proof, is much
%more complicated in the singular case (see \cite{Tak06, Eis10}), it seems to the author that some substantial generalizations of the arguments there are needed. 
%\end{rem}

From \eqref{kltizu}, we immediately get the following inequality:
\begin{cor}
There exists a constant $c_1=c_1(X,o)>0$ such that for any $v\in \Val_{X,o}$, we have:
\begin{equation}\label{izuvol}
A_X(v)^{n}\cdot \vol(v)\ge %\frac{1}{c^n}s(v)^{n}\cdot \vol(v)\ge 
 \frac{e(\fm)}{c_1^n}>0,
\end{equation}
where $e(\fm)$ is the Hilbert-Samuel multiplicity of $R$ along $\fm$.
\end{cor}
%\begin{rem}
%In \cite{Hu14}, the author introduced the notion of bounded homogeneous functions on the space of valuations. Inequality \eqref{izuvol} implies that $\vol^{-1/n}$ is a bounded homogenous function in the sense of \cite{Hu14}.  
%\end{rem}
\begin{rem}
We notice that an inequality closely related to \eqref{izuvol} appeared in \cite[Theorem 1.3]{dFM15}. However, the Mather version of log discrepancy was used there for general Cohen-Macaulay singularities. In this paper, we use the ordinary log discrepancy for any $\bQ$-Gorenstein klt singularities.
\end{rem}

\section{Minimizing normalized volume}\label{sec-minhvol}
%What normalization one should use?
%\begin{enumerate}
%\item
%$\nu(\mathfrak{m})=1$.
%\item
%Does the above correspond to $\mathcal{L}_v\Omega=n\Omega$.
%\item
%How is the normalization related to the weight function of Musta\c{t}\u{a}-Nicaise?
%\end{enumerate}
%If we have a blow up $\pi: Y\rightarrow X$ with an irreducible exceptional divisor $E$, then $\mathfrak{a}_m=\fa(\ord_E, m)=\pi_*(\mathcal{O}_Y(-m E))$.
%For a Kollar component, $(-(K_E+\Delta))^{n-1}=\vol_x(-(K_Y+E))={\rm vol}_x(A_X(E) E)={\rm vol}({\rm ord}_{A_X(E) E})={\rm vol}(A_X(E)^{-1}{\rm ord}_E)=A_X(E)^{n} {\rm vol}({\rm ord}_E)$.
From the above discussions in Section \ref{secpre}, it seems natural to ask whether we can minimize the rescaling invariant  (0-homogeneous) functional:
\begin{equation}
\hvol(v)=
\left\{
\begin{array}{ll}
A_X(v)^n {\rm vol}(v) & \mbox{ if } A_X(v)<+\infty;\\
+\infty & \mbox{ if } A_X(v)=+\infty.
\end{array}
\right.
\end{equation}

To answer this question, we would like some properness property. So we ask whether 

\[
\{v\in \Val_{X,o} ;  A_X(v)^n \vol(v)\le M\}
\]
is compact under appropriate topology (compare a similar problem in \cite[Proposition 5.9]{JM10}). Notice that if $A_X(v)=+\infty$, we have defined $\hvol(v)=+\infty$.

The main goal in this section is to answer this question positively. The main technical result is the following estimate:
\begin{thm}\label{mainthm}
Assume $(X,o)$ is a $\bQ$-Gorenstein klt singularity.
There exists a positive constant $c_2=c_2(X,o)>0$ such that for any valuation centered at $o$, we have:
\begin{equation}\label{keyineq}
\vol(v)\ge c_2 \left(\sup_{\fm}\frac{v}{\ord_o}\right)^{1-n}\frac{1}{v(\fm)}.
\end{equation}
In fact, we can choose $c_2=2^{-n}e(\mathfrak{m})$ where $e(\mathfrak{m})$ is the Hilbert-Samuel multiplicity of the maximal ideal $\fm$ defining $o\in X$.
\end{thm}

\begin{rem}\label{mainrem}
\begin{itemize}
\item 
Sebastien Boucksom and Mattias Jonsson pointed out to me that there are valuations such that the right hand side is zero. In this case, the inequality \eqref{keyineq} becomes trivial. However, if the log discrepancy $A(v)<+\infty$, %(such valution $v$ is called tempered in \cite{Hu14}), 
then Izumi-type inequality in Proposition \ref{Izumi} implies that the right hand side of the above inequality is strictly positive.
\item
$\sup_{\fm}(v/\ord_0)$ is called the ``Skewness'' function on $v$ in \cite{FJ04}. 
When $n=2$ and $o=0$ is smooth, by \cite[Remark 3.33]{FJ04} (see also \cite[Remark 4.9]{BFJ12}), there is an identity
\[
\vol(v)=\left(\sup_{\fm}\frac{v}{\ord_o}\right)^{-1}\frac{1}{v(\fm)}.
\]
So inequality \eqref{keyineq} is a weak generalization of this formula to higher dimension. On the other hand, in higher dimensions, as we will see in Examples \ref{exmono}, there is no inequality in the other direction, i.e. there is no uniform upper bound of $\vol(v)\left(\sup_{\fm}v/\ord_o\right)^{n-1} v(\fm)$.
%\end{itemize}
%\end{rem}
%\begin{examp}
%When $n=2$ and $o=0$ is smooth, by \cite[Remark 3.33]{FJ04} (see also \cite[Remark 4.9]{BFJ12}), we have
%\[
%\vol(v)=\left(\sup_{\fm}\frac{v}{\ord_o}\right)^{-1}\frac{1}{v(\fm)}\Longrightarrow A(v)^2\vol(v)\ge \frac{A(v)}{\sup_{\fm}(v/\ord_o)}\frac{A(v)}{v(\fm)}\ge \frac{A(v)}{v(\fm)}.
%\]
%The last inequality is by Izumi's theorem in \cite[Proposition 5.10]{JM10}.
%\end{examp}
%Combining the estimate \eqref{keyineq} and Izumi-type inequality in Proposition \ref{Izumi}, we can deduce the following theorem.
%\begin{thm}[Properness]\label{proper}
%Assume $(X,o)$ is a $\bQ$-Gorenstein klt singularity.
%There exists a positive constant $c=c(X,o)>0$ such that for any valuation $v$ centered at $o\in X$ with $A(v)<+\infty$, we have:
%\begin{equation}\label{properineq}
%A_X(v)^n\cdot \vol(v)\ge c\frac{A_X(v)}{v(\fm)}.
%\end{equation}
%\end{thm}
%\begin{rem}\label{reminfld}
\end{itemize}
\end{rem}
%Now notice that the proof of compactness result in \cite[Proposition 5.9]{JM10} is also valid in the singular case (cf. \cite{BFFU13}, \cite{MN}). 
\begin{cor}
Theorem \ref{proper} holds.
\end{cor}
\begin{proof}
Combining \eqref{keyineq} and Izumi type estimate \eqref{kltizu}, we get the estimate in Theorem \ref{proper} with $K=c_2 c_1^{1-n}$:
\begin{eqnarray*}
A_X(v)^n \cdot \vol(v)&\ge & c_2 \left(\sup_{\fm} \frac{v}{A_X(v) \ord_o}\right)^{1-n}\frac{A_X(v)}{v(\fm)}\\
&\ge & c_2 c_1^{1-n} \frac{A_X(v)}{v(\fm)}.
\end{eqnarray*}

\end{proof}
Now following Jonsson-Musta\c{t}\u{a} (\cite[Section 4.1]{JM10}), we endow the space $\Val_{X,o}$ with weakest topology for which the evaluation map $\Val_{X,o}\ni v\rightarrow v(f)$ is continuous for all nonzero rational function $f$ in the quotient ring of $R=\cO_{X,o}$. 
To continue, we make the following:
\vskip 2mm
{\bf Hypothesis C:} $\hvol$ is lower semicontinuous on $\Val_{X,o}$.
\vskip 2mm

By Theorem \ref{proper} we get:
\begin{cor}\label{maincor}
Assume $(X, o)$ is a $\bQ$-Gorenstein klt singularity. Under {\bf hypothesis C}, the normalized volume functional $\hvol(v)$ is minimized at a real valuation $v_*\in \Val_{X,o}$ (and hence at any positive multiple of $v_*$).
\end{cor}
\begin{proof}
By inequality \eqref{izuvol}, $\hvol$ has a positive lower bound on $\Val_{X,o}$. For convenience, we denote:
\begin{equation}\label{eq-hvX}
\hvol(X,o):=\inf_{v\in \Val_{X,o}}\hvol(v).
\end{equation}
Notice that both sides of \eqref{properineq} are rescaling invariant. So we can restrict to consider only valuations with $v(\fm)=1$. By the estimate \eqref{properineq} for any $C>0$, the sub level set 
\[
W_C:=\{v\in \Val_{X}; c_X(v)=o, v(\fm)=1, \hvol(v)\le C\}
\] 
is contained in the following set,
\[
V_M:=\{v\in \Val_X; c_X(v)=o, v(\fm)=1, A(v)\le M\}
\]
for some uniform $M=M(C,X,o)>0$. Now using the same proof of \cite[Proposition 5.9]{JM10}, we know that $V_M$
is compact subspace of $\Val_{X,o}$. By letting $C\rightarrow \hvol(X,o)>0$ (see \eqref{eq-hvX}) and using {\bf Hypothesis C}, we conclude that there is a minimizer 
\[
v_*\in \bigcap_{C\ge \hvol(X,o)}W_C.
\] 
\end{proof}

\begin{rem}\label{remcont}
%We consider {\bf Hypothesis C} as provisional and expect it to be true (see Conjecture \ref{mainconj}). 
By \cite[Lemma 5.7]{JM10}, $A_X(v)$ is lower semicontinuous. So if $\vol$ is continuous then {\bf Hypothesis C} is valid. Note that by \cite[Corollary D]{BFJ12}, $\vol$ is continuous on the space of quasi-monomial valuations for a fixed log smooth model. So for any fixed log smooth model $(Y,D)$, $\hvol$ has a minimizer when restricted to ${\rm QM}(Y,D)$. Actually we expect that the minimizer $v_*$ is obtained on some log smooth model (item 3 of Conjecture \ref{mainconj}).
\end{rem}

\subsection{Proof of Theorem \ref{mainthm}}\label{sec-pfmain}
Denote by $R=\mathcal{O}_{X,o}$ the local ring at point $o\in X$. The maximal ideal $\fm$ 
consists of all holomorphic germs vanishing at $o$. Then $R$ is a finitely generated $\bC$ algebra
and $R/\fm=\bC$.

Let $v$ be a valuation centered at $\fm$. From now on we normalize 
$v$ such that $v(\fm)=1$.
We will denote $v_0=\ord_o$. Notice that $v_0$ does not have to be a valuation. However it satisfies the following properties:
\begin{enumerate}
\item $v_0(fg)\ge v_0(f)+v_0(g)$. The equality does not hold if ${\rm Gr}_m R$ is not integral (for example: 2-dimensional $A_2$-singularity: $R=A_{\mathfrak{m}}$ where $A=\bC[x,y,z]/(x^2+y^2+z^3)$ and $\mathfrak{m}=(\bar{x},\bar{y},\bar{z})$). 
\item $v_0(f+g)\ge \min\{v_0(f), v_0(g)\}$ and the equality holds if $v_0(f)\neq v_0(g)$. 
\end{enumerate}
We define the following invariant of $v$:
\begin{equation}\label{eqsv2}
s(v)=\max\left\{2, \left\lceil\sup_{\fm}\frac{v}{v_0}\right\rceil\right\}.
\end{equation}
It's easy to see that (using $\sup_{\fm}(v/v_0)\ge 1$)
\begin{equation}\label{eq-sv2}
s(v)\le 2\sup_{\fm}(v/v_0).
\end{equation}
By Theorem \ref{Izumi}, $s=s(v)<\infty$ so that 
\begin{equation}\label{izumineq1}
v_0=v(\fm)v_0\le v\le s(v) v_0.
\end{equation} 
Denote the valuation ideals $
\fa_r:=\{f\in R; v(f)\ge r\}$. 
%We will normalize $v$ such that $v(\fm)=1$. 
Then by \eqref{izumineq1}, we have
\begin{equation}\label{inclusion}
\fm^r\subseteq \fa_r\subseteq \fm^{\lfloor r/s\rfloor}.
\end{equation}
 By \eqref{inclusion} we have:
\[
\len(R/\fm^{r})\ge \len(R/\fa_r)\ge \len(R/\fm^{\lfloor r/s\rfloor}).
\]
Observe that in our situation we can replace $\len$ by $\dim_{\bC}$, because $\fa_r$ is $\fm$-primary. 
%\[
%\len(R/\fm^{r+1})=\sum_{i=0}^{n} \len(\fm^{i}/\fm^{i+1})=\sum_{i=0}^n\dim_{\bC}(\fm^i/\fm^{i+1})=\dim_{\bC}(R/\fm^{r+1}).
%\]
%Here we used that $\fm^{i}/\fm^{i+1}$ is annihilated by $\fm$ so is naturally a $R/\fm=\bC$ module, i.e. a $\bC$ vector space.
Denote 
%\[
%F(n)=\len(R/\fm^r), \quad G(n)=\len(R/\fa_r).
%\] 
\begin{eqnarray*}
d_r:=\len(\fm^{r}/\fm^{r+1})&=&\dim_{\bC}(\fm^r/\fm^{r+1})=\dim_\bC (R/\fm^{r+1})-\dim_\bC(R/\fm^{r}). %=F(n+1)-F(n).
\end{eqnarray*}
It's well known that:
\[
d_r=e(\fm) \frac{r^{n-1}}{(n-1)!}+O(r^{n-2})
\]
where $e(\fm)$ is the Hilbert-Samuel multiplicity of $\fm$.

%\begin{prop}
%There exists a constant $c=c(X,o)$ such that for any valuation centered at $o$, we have:
%\[
%\vol(v)\ge c \left(\sup_{\fm}\frac{v}{\ord_o}\right)^{1-n}\frac{1}{v(\fm)}.
%\]
%\end{prop}

\begin{proof}[Proof of Theorem \ref{mainthm}]
First normalize $v$ such that 
\[
1=v(\fm)=\inf_{\fm}\{v(f); f\in\fm\}=v(g),
\] 
for some fixed $g\in \fm$.
We will use the notation at the beginning of this section. In particular, $s(v)$ in \eqref{eqsv2} satisfies $s(v)\le 2\sup_{\fm}(v/v_0)$. 
So we only need to prove the following inequality:
\begin{equation}\label{normalineq}
\vol(v)\ge c\cdot s(v)^{1-n},
\end{equation}
for some uniform positive constant $c$.
For any $l, r\in \mathbb{N}$ with $0\le l\le r$, we denote $p=p(l,s)=\lfloor l/s\rfloor$. Choose a $\bC$-basis of $\fm^{p}/\fm^{p+1}$:
\[
[{\bf u}^{(p)}]:=[\{u^{(p)}_1,\dots, u^{(p)}_{d_p}\}]=\{[u^{(p)}_1],\dots, [u^{(p)}_{d_p}]\}.
\]
We have: $v_0(u^{(p)}_i)=p$ for any $1\le i\le d_p$. By \eqref{izumineq1}, we have:
\[
p\le v(u_i^{(p)})\le s(v)p.
\]
Now we define the elements:
\[
x_i^{(r-l, p)}=g^{r-l}\cdot u^{(p)}_i, \mbox{ for } 1\le i\le d_p.
\]
Then $v_0(x_i^{(r-l,p)})\ge r-l+p$ and more importantly
\[
v(x_i^{(r-l,p)})=r-l+v(u_i^{(p)})\le r-l+sp=r-l+s\lfloor l/s\rfloor\le r.
\]
Now we can arrange the following table. Notice that to prove our result, we only need to consider those $r$'s which are integral multiples of $s$.
\begin{center}
\renewcommand*{\arraystretch}{1.6}
\begin{tabular}{c|c|c|c|c}
$p$&$l$ & function & lower bound of $\ord_o$ & range of $v$ \\
\thickhline
$0$& $0$     &    $x^{(r,0)}_1=g^{r}$        & $r$ & $r$ \\
$0$&$1$     &    $x^{(r-1,0)}_1=g^{r-1}$    & $r-1$ & $r-1$\\
\dots&\dots   &   \dots                                                 &  $\dots$ & $\dots$\\
$0$&$s-2$     &    $x^{(r-s+2,0)}_1=g^{r-s+2}$   &  $r-s+2$ & $r-s+2$\\
\hline
$1$&$s$ &    $x^{(r-s,1)}_1, \dots, x^{(r-s, 1)}_{d_1}$ & $r-s+1$ & $[r-s+1, r]$\\
$1$& $s+1$ & $x^{(r-s-1,1)}_1, \dots, x^{(r-s-1,1)}_{d_1}$ & $r-s$ & $[r-s, r]$\\
\dots & \dots &\dots &\dots\\
$1$&$2s-2$ & $x^{(r-2s+2,1)}_1, \dots, x^{(r-2s+2,1)}_{d_1}$ & $r-2s+3$ & $[r-2s+3, r]$ \\
\hline
$2$&$2s$ & $x^{(r-2s,2)}_1, \dots, x^{(r-2s,2)}_{d_2}$ & $r-2s+2$ & $[r-2s+2, r]$\\
\dots &\dots &\dots &\dots\\
%\cdashline{1-5}
%\multirow{2}{*}{$q-1$} & \multirow{2}{*}{$(2q-1)s-2$} & 
%$x^{n-(2q-1)s+2,2q-2}_1, \dots$ &\multirow{2}{*}{$r-(2q-1)s+2q$} & $[r-(2q-1)s+2q,$ \\
%& & $x^{n-(2q-1)s+2,2q-2}_{d_{2q-2}}$ & & \quad\quad $r-(2q-1)s+2q+1]$\\
%\hline
$q-1$&$qs-2$ & $x^{(r-qs+2,q-1)}_1, \dots, x^{(r-qs+2,q-1)}_{d_{q-1}}$ & $r-qs+q+1$ & 
$[r-qs+q+1, r]$\\
\hline
$q$&$q s$ & $x^{(r-q s, q)}_1, \dots, x^{(r-q s, q)}_{d_{q}}$ & $r-q s+q$ & $[r-q s+q, r]$\\
\dots&\dots &\dots &\dots &\dots\\
%\cdashline{1-5}
$\frac{r}{s}-1$ & $r-2$ & $x_1^{(2,(r/s)-1)}, \dots, x_{d_{(r/s)-2}}^{(2,(r/s)-1)}$ & $(r/s)+1$ & $[(r/s)+1, r]$\\
\hline
$r/s$ & $r$ & $x_1^{(0,r/s)},\dots, x_{d_{r/s}}^{(0,r/s)}$ & $r/s$ & $[r/s, r]$
\end{tabular}
\end{center}
We consider the set:  
\[
\left\{\left[x_i^{(r-l,\lfloor l/s\rfloor)}\right]\in R/\fa_{r+1}; \text{ for all } x_i^{(r-l,\lfloor l/s\rfloor)} \text{ in the 3rd column of the above table}\right\}.
\]
{\bf Claim:} The above set consists of linearly independent elements.\\

Assuming the claim, we can get the estimate of $\dim_{\bC}R/\fa_{r+1}$:
\begin{eqnarray*}
\dim_{\bC}R/\fa_{r+1}&\ge &(s-1)(d_0+d_1+\dots+d_{(r/s)-1})\\
&=&(s-1)\sum_{i=0}^{(r/s)-1}\left[ e(\fm)\frac{i^{n-1}}{(n-1)!}+O(i^{n-2})\right]\\
&=&\frac{ e(\fm)(s-1)}{(n-1)!}\frac{1}{n}\left(\frac{r}{s}\right)^{n}+O(r^{n-1})\\
&=&\frac{e(\fm)(s-1)}{s^{n}}\frac{r^n}{n!}+O(r^{n-1}).
\end{eqnarray*}
So we can estimate the lower bounds of volume:
\begin{eqnarray*}
\vol(v)&=&\lim_{r\rightarrow+\infty}\frac{n!}{(r+1)^n}\dim_{\bC} R/\fa_{r+1}\ge e(\mathfrak{m})\left(s^{1-n}-s^{-n}\right).
\end{eqnarray*}
By the definition of $s(v)$ in \eqref{eqsv2}, $s\ge 2$ so that $(s-1)/s^n\ge s^{1-n}/2$. Combining this with \eqref{eq-sv2}, we get:
\begin{eqnarray*}
\vol(v)\ge e(\fm) s^{1-n}/2= 2^{-n} e(\fm) \left(\sup_{\fm}\frac{v}{v_0}\right)^{1-n}.
\end{eqnarray*}
So we complete the proof of Theorem \ref{mainthm} with $c_2=2^{-n} e(\fm)$. % \eqref{normalineq} from the definition of volume \eqref{defvol}. 

Finally we verify the Claim. Firstly, elements coming from the same row are linearly indepenent. Indeed, for any ${\bf a}=\{a_1,\dots, a_{d_p}\}\neq {\bf 0}$, we have:
\[
x=\sum_{i=1}^{d_p}a_i x_i^{(r-l,p)}=g^{r-l}\sum_{i=1}^{d_p}a_i u_i^{(p)}.
\]
Since $\{u_i^{(p)}; 1\le i\le d_p\}$ is a basis of $\fm^p/\fm^{p+1}$, we see that
\[
u^{(p)}:=\sum_{i=1}^{d_p}a_i u_i^{(p)}\in \fm^p-\fm^{p+1}. 
\]
So $\ord_o(u^{(p)})=p=\lfloor l/s\rfloor$ and $v(x)=r-l+v(u^{(p)})\le r-l+\lfloor l/s\rfloor s\le r$, which implies $[x]\neq 0\in R/\fa_{r+1}$.

Next, we show that all elements are linearly independent. Notice that for any linear combination $x$ of all $x^{(r-l,\lfloor l/s\rfloor)}_{i}$ in the 3rd column of the above table, we can decompose it as:
\[
x=\sum_{l=0, s\nmid (l+1)}^{r}  b_{l} x^{(r-l,\lfloor l/s\rfloor) }=\sum_{l=0,s\nmid (l+1)}^{r} b_{l}g^{r-l}u^{(\lfloor l/s\rfloor)}
=g^{r-L}\sum_{l=0, s\nmid (l+1)}^{L} b_l g^{L-l}u^{(\lfloor l/s\rfloor)}=:g^{r-L}y,
\]
where $L=\max\{l; b_l\neq 0\}$ satisfies $s\nmid L+1$, and $u^{(\lfloor l/s\rfloor)}$ is a nontrivial linear combination of $u^{(\lfloor l/s\rfloor)}_1, \dots, u^{(\lfloor l/s \rfloor)}_{d_{\lfloor l/s\rfloor}}$. Now it's clear that (compare with the above table)
\[
v_0(g^{L-l}u^{(\lfloor l/s\rfloor)})\ge L-l+\lfloor l/s\rfloor > \lfloor L/s\rfloor,
\]
for any $l<L$ and $s\nmid (l+1)$. So we have:
\[
v_0(y)=\min\left\{v_0(g^{L-l}u^{(\lfloor l/s\rfloor)}); b_l\neq 0\right\}=v_0(u^{(\lfloor L/s\rfloor)})=\lfloor L/s\rfloor.
\]
So we see that
\[
v(x)=r-L+v(y)\le r-L+s\cdot v_0(y)=r-L+s \lfloor L/s\rfloor\le r,
\]
which implies $[x]\neq 0\in R/\fa_{r+1}$.
\end{proof}

\section{Examples}\label{SecExmp}

\begin{examp}

\label{exmono}

%When $(X, o)=(\bC^n, 0)$, there is a canonical valuation $\ord_{\bP^{n-1}}$ where $\bP^{n-1}$ is the exceptional divisor. 
We consider monomial valuations on $(\bC^n, 0)$.
%(and monomial valuations on toric singularities). 
For ${\bf x}=(x_1,\dots, x_n)\in \bR^n_{>0}$, without loss of generality, we can assume $x_1\le x_2\le \dots\le x_n$.
\[
A_X(v_{\bf x})=\sum_{i}x_i, \vol(v_{\bf x})=\frac{1}{\prod_{i}x_i}, v_{\bf x}(\fm)=\min_i\{x_i\}=x_1, \sup_{\fm}\frac{v}{\ord_o}=\sup_i\{x_i\}=x_n.
\]
So we get that:
\[
\left(\sup_{\fm}\frac{v}{\ord_o}\right)^{n-1}\vol(v_{\bf x})v_{\bx}(\fm)=\frac{x_n^{n-1}x_1}{\prod_{i=1}^nx_i}=\prod_{i=2}^{n-1} \frac{x_n}{x_i}\ge 1.
\]
However, there is no upper bound for the right hand side as mentioned in Remark \ref{mainrem}.

We can also observe the inequalities:
\begin{enumerate}
\item
\[
\sup_{\fm}\frac{v}{\ord_o}=x_n\le \sum_{i=1}^nx_i= A(v_{\bf x}).
\]
\item
\[
A_X(v_{\bf x})^n\vol(v_{\bf x})=\frac{(\sum_i x_i)^n}{\prod_i x_i}\ge\frac{\sum_i x_i}{\min_i\{x_i\}}= \frac{A_X(v_{\bf x})}{v_{\bf x}(\fm)}
\]
\end{enumerate}
By using arithmetic-geometric inequality, we also get the estimate:
\[
\vol(v_{\bf x})=\frac{1}{\prod_i x_i}\ge \frac{n^n}{(\sum_i x_i)^n}.
\]
Notice that
\[
\lct(\fa_r(v))=\frac{1}{r}\sum_{i}x_i=\frac{A_X(v_{\bf x})}{r}=\frac{A_X(v_{\bf x})}{v_{\bf x}(\fa_r(v_{\bf x}))}\Longrightarrow \lct(\fa_{\LargerCdot}(v_{\bf x}))=\sum_i x_i.
\]
So we get:
\[
\vol(v_{\bf x})\ge \frac{n^n}{\lct(\fa_{\LargerCdot}(v_{\bf x}))}.
\]
This is essentially the monomial case of more general results from \cite{FEM}, \cite{Mus02}, which says that:
\begin{equation}\label{eq-FEM}
e(\fa_{\LargerCdot}) \ge \frac{n^n}{{\rm lct}(\fa_{\LargerCdot})^n}
\end{equation}
for any graded sequence of zero-dimensional ideals $\fa_{\LargerCdot}$ in $R$. Applying \eqref{eq-FEM} to $\{\fa_r(v)\}$ we get:
\[
\vol(v)\ge \frac{n^n}{{\rm lct}(\fa_\LargerCdot(v))^{n}}\ge \frac{n^n}{A_X(v)^{n}}=\frac{A_X(\ord_0)^n\vol(\ord_0)}{A_X(v)^n}.
\]
In other words, we have proved the following:
\begin{prop}\label{prop-sm}
Let $(X, o)=(\bC^n, 0)$ and $E=\bP^{n-1}$ be the exceptional divisor of the standard blow up at $0\in \bC^n$. 
Then $\ord_{E}$ is a minimizer of $\hvol$ on $\Val_{X,o}$ with $\hvol(\ord_E)=n^n$.
\end{prop}
%In other words, $\hvol(v)$ is (globally) minimized at the $v_*=\ord_0$. This proves Theorem \ref{Thm-Cn}. 
%\begin{rem}
%We emphasize that this example shows that the minimization problem here can be seen as seeking for generalization of arithmetic-geometric inequality to a much more broad setting.
%\end{rem}

\end{examp}

\begin{examp}[experimental example]\label{exhyper}

Assume that $X^n=\{f(z_1, \dots, z_{n+1})=0\}\subset\bC^{n+1}$ is a hypersurface of dimension $n\ge 3$ with isolated $\bQ$-Gorenstein klt singularities at $0$. For any $\bx=(x_1, \dots, x_{n+1})\in \mathbb{Q}_{+}^{n+1}$, we can rescale it to an $n$-tuple of integers $\tilde{\bx}\in \bZ_{+}^{n+1}$.  Since the normalized volume is rescaling invariant, we will identify $\bx$ with $\tilde{\bx}$ in this example.
hence it determines a weighted blow-up of $\pi=\pi^{\bx}: Z=\bC_\bx \rightarrow \mathbb{C}^{n+1}$ with exceptional divisor $F$. $F$ is the weighted projective space $\mathbb{P}({\bf x})$. Let $Y=X_{\bf x}$ be the strict transform of $X$ and $\pi|_Y: Y\rightarrow X$ is $v_{\bf x}$-blow up with exceptional divisor $E=F\cap Y$. 
%If ${\bf x}$ is rational, we can rescale to assume that $(x_1, \dots, x_{n+1})$ is $n$-tuple of integers. 
Then $E$ is the hypersurface of weighted degree $v_{\bf x}(f)$ in $\mathbb{P}({\bf x})$. 

$v_{\bf x}$ induces a valuation on $k(X_{\bf x})$ for which we still denoted by $v_{\bf x}$:
\[
v_{\bf x}(f)=\min_{\tilde{f}|_{X_{\bf x}}=f, \tilde{f}\in k(\bC_{\bf x}^{n+1})} v_{\bf x}(\tilde{f}), \quad f\in \mathbb{C}(\bC^{n+1}_{\bf x})
\]
By \cite[Lemma 2.2]{Mar96}, we have: $v_{\bf x}(f)=\lfloor \frac{1}{d_{\Gamma}}\ord_{\Gamma}(f) \rfloor$ where $\Gamma$ is any component of $E=X_{\bf x}\cap F$ of codimension 2 in $F$ such that $X_{\bf x}$ is normal at the generic point of $\Gamma$.

Remember that we can assume $\bx\in \bZ_{+}^{n+1}$ by rescaling. It's easy to see that we have:
\[
K_{Z}+F=\pi^*K_{\bC^{n+1}}+\left(\sum_i x_i\right) F, \quad Y=\pi^*(X)-v_{\bf x}(f) F.
\]
So we get:
\[
K_Z+Y+F=\pi^*(K_{\bC^{n+1}}+X)+\left((\sum_i x_i)-v_{\bf x}(f)\right)F
\]
Taking adjunction, we get:
\[
K_Y+{\rm Diff}_Y(F)=(\pi|_Y)^* K_{X}+A_X(v_{\bf x}) F|_Y.
\]
There are at least two cases one can determine the log discrepancy of the $v_{\bx}$:
\begin{itemize}
\item
If $Y$ is a Cartier divisor in $Z=\bC^{n+1}_{\bx}$, ${\rm Diff}_{Y}(F)=F|_Y$ by \cite[Proposition 4.2]{Kol13}.
\item 
If $\Gamma$ above is not a toric variety, then $A_X\left(\frac{1}{d_\Gamma}\ord_\Gamma\right)=\sum_i x_i -v_{\bx}(f)$ by \cite[Proposition 2.3]{Mar96}
\end{itemize}
The point is that $A_X(v_\bx)=\sum_i x_i-v_{\bx}(f)$ should hold for generic $\bx\in \bQ_{+}^{n+1}$. So in this experimental example, we will (continuously) extend this expression to any $\bx\in \bR_{+}^{n+1}$ 
and calculate the minimum of $\hvol(v_{\bx})$ as a function of $\bx\in \bR_{+}^{n+1}$.
\begin{enumerate}
\item
Consider the $n$-dimensional $A_1$ singularity:
\[
X^n=\{z_1^2+z_2^2+\dots+z_n^2+z_{n+1}^2=0\}\subset \mathbb{C}^{n+1}.
\]
By renaming the variables, we consider the weights $\bx\in \bR_{+}^{n+1}$ satisfying $
x_1\le x_2\le \dots \le x_n$. 
\[
A(v_{\bf x})=x_1+\dots+x_{n+1}-2x_1, \quad 
\]
Modulo $O(r^{n-1})$, it's easy to see that:
\[
\dim_{\bC}R/\fa_r(v_{\bf x})=\frac{1}{\prod_{i=1}^{n+1} x_i}\frac{r^{n+1}}{(n+1)!}-\frac{1}{\prod_{i=1}^{n+1} x_i}\frac{(r-2x_1)^{n+1}}{(n+1)!}=\frac{2x_1}{\prod_{i=1}^{n+1}x_i}\frac{r^n}{n!}.
\]
So $\vol(v_{\bf x})=\frac{2}{\prod_{i=2}^{n+1} x_i}$ and we get:
\begin{eqnarray*}
A_X(v_{\bf x})^n\vol(v_{\bf x})&=&\frac{2(-x_1+\sum_{i=2}^{n+1}x_i)^n}{\prod_{i=2}^{n+1}x_i}\ge 2n^n\left(\frac{-x_1+\sum_{i=2}^{n+1}x_i}{\sum_{i=2}^{n+1}x_i}\right)^n\\
&\ge& 2n^n \left(\frac{-x_1+n x_1}{n x_1}\right)^n=2(n-1)^n.
\end{eqnarray*}
So $\hvol(v_{\bf x})$ obtains the minimum $2(n-1)^n$ at $x=(1,\dots, 1)$.

\item
Consider the $n$-dimensional $A_{k-1}$ singularity ($n\ge 2$):
\[
X^n=A^{n}_{k-1}:=\{z_1^2+z_2^2+\dots+z_n^2+z_{n+1}^{k}=0\}\subset \mathbb{C}^{n+1}.
\]
Again for ${\bf x}=(x_1, \dots, x_{n+1})\in \bR^{n+1}$, we can assume $0<x_1\le x_2\le \dots\le x_{n+1}$. So we get:
\[
A(v_{\bf x})=\sum_{i=1}^{n+1} x_i -\min\{2 x_1, k x_{n+1}\}.
\]
There are two cases to consider.
\begin{enumerate}
\item $2x_1\le k x_{n+1}$. %We define $1\le j\le n$ to be the largest index such that $2 x_j\le k x_{n+1}$. 
Modulo $O(r^{n-1})$ we have:
\[
\dim_{\bC} R/\fa_r(v_{\bf x})=\frac{1}{\prod_{i=1}^{n+1}x_i}\frac{r^{n+1}}{(n+1)!}-\frac{1}{\prod_{i=1}^{n+1} x_i}\frac{(r-2x_1)^n}{(n+1)!}=\frac{2x_1}{\prod_{i=1}^{n+1} x_i}\frac{r^n}{n!}.
\]
So $\vol(v_\bx)=2/\prod_{i=2}^{n+1}x_i$ and using arithmetic-geometric mean inequality we can estimate:
\begin{eqnarray*}
\hvol(v_\bx)&=&\left(x_{n+1}+\sum_{i=2}^{n} x_i-  x_1\right)^n \frac{2}{ x_{n+1} \prod_{i=2}^{n} x_i}\\
%&\ge &\left(x_{n+1}+\sum_{i=1,i\neq j}^n x_i -x_1\right)^n\frac{2}{x_{n+1}\prod_{i=1, i\neq j}^n x_i}\\
&\ge& 2(n-1)^{n-1} \left(\frac{x_{n+1}+\sum_{i=2}^n x_i - x_1}{\sum_{i=2}^nx_i}\right)^n\frac{\sum_{i=2} x_i}{x_{n+1}}.
\end{eqnarray*}
If we define $\alpha=(n-1) x_{n+1}/(\sum_{i=2}^n x_i)$ and $\tau=(n-1)x_1/(\sum_{i=2}^n x_i)\in (0, 1)$, then 
\[
\hvol=2\frac{ (\alpha+(n-1)-\tau)^n}{\alpha}\ge 2\frac{\left(\alpha+n-2\right)^n}{\alpha}=:\phi(\alpha).
\]
It's easy to verify that $\phi(\alpha)$ obtains its minimum $\frac{2n^n(n-2)^{n-1}}{(n-1)^{n-1}}$ at $\alpha_*=\frac{n-2}{n-1}$, is strictly decreasing on $(0, \alpha_*)$ and strictly increasing on $(\alpha_*, +\infty)$. Tracking the equality case of the above estimates, we get that $\hvol$ obtains $\phi(\alpha_*)$ only if ${\bf x}_*=(1, 1,\dots, 1, \frac{n-2}{n-1})$. The constraint is satisfied at the minimizer $\bx_*$ if and only if $2\le k\frac{(n-2)}{n-1}$.

\item $2 x_1\ge k x_{n+1}$. In this case, 
\[
A(v_\bx)=\sum_{i=1}^{n}x_i +x_{n+1}-k x_{n+1}, \quad 
\vol(v_{\bf x})=\frac{k}{\prod_{i=1}^{n}x_i}.
\]
So
\[
\hvol(v_\bx)=\frac{k(\sum_{i=1}^n x_i -(k-1) x_{n+1})^n}{\prod_{i=1}^n x_i}\ge  k n^n \left(\frac{\sum_{i=1}^n x_i-(k-1)x_{n+1}}{\sum_{i=1}^n x_i}\right)^n.
\]
The last expression is a decreasing function of $x_{n+1}$ and hence obtains the minimum at $x_{n+1}=2x_1/k$ which is the right end point of the interval $(0,2x_1/k)$ for $x_{n+1}$ in this case. So $\hvol$ obtains its minimum %$n^n k\left(1-\frac{2(k-1)}{kn}\right)^n=\phi(2/k)$ 
$\frac{\left((n-2)k+2\right)^n}{k^{n-1}}=\phi(2/k)$
at the weight $(1, 1, \dots, 1, \frac{2}{k})$ under the constraint $2 x_1\ge k x_{n+1}$.  
\end{enumerate}

Combining the above discussions, we see that if $2\le k\frac{n-2}{n-1}$, then $\hvol(v_{\bx})$ obtains the minimum $\phi(\alpha_*)$ at weight $(1,\dots, 1, \frac{n-2}{n-1})$. Otherwise, the minimum is $\phi(2/k)$ and is obtained at the weight $(1,\dots, 1, 2/k)$. We can organize the situation in the following table. %See also Figure \ref{figAk} for the illustration for this minimization problem.
\vskip 2mm
\begin{center}
\renewcommand*{\arraystretch}{1.5}
\begin{tabular}{|ctc|c|c|c}
\hline
%\theadfont
\diagbox[width=2.5em]{$k$\quad}{$n$}
%&\thead{Second\\column}&\thead{Third\\column}
& $2$ & $3$ & $4$ & $n\ge 5$ \\
\thickhline
$1$ & $(1,1,2)$     &    $(1,1,1,2)$      &  $(1,1,1,1,2)$  & $(1,\dots, 1, 2)$ \\
$2$& $(1,1,1)$     &    $(1,1,1,1)$      &  $(1,1,1,1,1)$  & $(1, \dots, 1, 1)$ \\  \cline{5-5}
$3$& $(1,1,2/3)$  &    $(1,1,1,2/3))$   & $\boxed{(1,1,1,1,2/3)}$ & $(1, \dots, 1, \frac{n-2}{n-1})$ \\ \cline{4-4}
$4$& $(1,1,1/2)$ &     $\boxed{(1,1,1,1/2)}$    & $(1,1,1,1,2/3)$ & $(1, \dots, 1, \frac{n-2}{n-1})$ \\ \cline{3-3}
$k\ge 5$& $(1,1, 2/k)$ & $(1,1,1, 1/2)$ & $(1,1,1,1,2/3)$ & $(1, \dots, 1, \frac{n-2}{n-1})$ \\
\hline
\end{tabular}
\end{center}
\vskip 2mm
Surprisingly, this table matches a table (pointed to us by H-J. Hein) coming out of the study of Ricci-flat K\"{a}hler cone metrics or the closely related Sasaki-Einstein metrics. This 
motivates one of our conjectures in Conjecture \ref{mainconj}. 
For comparison and for the reader's convenience, we give some discussion analogous to that in \cite[Section 3.3]{DS15}. We denote by $A^n_{k-1}$ the $n$-dimensional 
$A_{k-1}$ singularities considered above. Then we have the following cases:
%in \cite{GMSY07, LS, Li1, HN} (see the discussion in \cite[Section 3.3]{DS15}).
%More precisely, for all the weights in the 1st column and those weights above the boxed weights, there are Sasaki-Einstein metrics on the corresponding
%$n$-dimensional $A_{k-1}$ singularity with the weights being the correct Reeb vector field. We 
\begin{enumerate}
\item $n=2$. For any $k\ge 1$, $A^2_{k-1}\cong \bC^2/\mathbb{Z}_{k-1}$ has a flat quotient metric. 

\item $k=1$.  For any $n\ge 2$, $A^n_0=\bC^n$ has the standard flat metric. 

\item $k=2$. For any $n\ge 2$, $A^n_1$ has a rotationally symmetric Ricci-flat cone metric which could be obtained by solving a simple
ODE using Calabi's ansatz. 

\item $(n,k)=(3,3)$. There is a Ricci-flat cone K\"{a}hler metric with the weight of the $\bC^*$-action given by $(1,1, 1, 2/3)$. 
This is nontrivial and was first proved in \cite{LS14} (see also \cite{Li15}).

\item $(n,k)=(3,4)$ or $(4,3)$. For these two special weights (boxed weights in the table), the corresponding singularities are K-semistable but not K-polystable. More precisely, if $D^2$ (resp. $D^3$) denotes the smooth
quadric hypersurface in $\bP^{2}$ (resp. $\bP^3$), then the associated pair $(\bP^{n-1}, (1-\frac{1}{k})D)$  
is log-K-semistable but not log-K-polystable by \cite{LS14, Li13}. 

\item For all the other cases (corresponding to all the weights below the short horizontal segments), we have $2<k\frac{n-2}{n-1}$. It was shown in \cite{GMSY07} that there is no (obvious) Sasaki-Einstein metric on the corresponding $A^n_{k-1}$ singularity.
On the other hand, Hein-Naber recently showed in \cite{HN} that there is a local (non-conic) Ricci-flat metric near the vertex $o\in A^n_{k-1}$ whose metric tangent cone at $o$ is $A^{n-1}_1\times \bC$ with
the product Ricci flat metric. The K\"{a}hler potential of the Ricci flat metric is given by 
$\left(|z_1|^2+\cdots+|z_n|^2\right)^{\frac{n-2}{n-1}}+|z_{n+1}|^2$.
It was then observed in \cite{DS15} that the weights $(\frac{n-1}{n-2}, \dots, \frac{n-1}{n-2}, 1)=\frac{n-1}{n-2}(1,\dots, 1, \frac{n-2}{n-1})$ indeed degenerates $A^n_{k-1}$ into $A^{n-1}_1\times \bC$ in these cases, i.e. when $k>2\frac{n-1}{n-2}$.

\end{enumerate}

In the Appendix I, we calculate the candidate minimizing valuations for $D$-type and $E$-type singularities
\end{enumerate}

\end{examp}

\section{Relation to Fujita's divisorial stability }\label{secFujita}

In this section, we carry out calculations to show that there is a close relation between minimization of $\hvol$ with Fujita's divisorial stability (\cite{Fuj15a}) which is a consequence of K-stability (\cite{Tia97}, \cite{Don02}, see also \cite{Ber12}). The calculations will essentially show that the derivative of $\hvol$ at the canonial valuation on the cone along some  directions of $\bC^*$-invariant valuations is given by Fujita's invariant on the base. Note that Fujita's invariant is an example of CM weight (\cite{Tia97}, \cite{Don02}) which is a generalization of the orginal Futaki invariant (\cite{Fut83}) to the setting of degenerations. So our calculation is a reflection of the calculation of Martelli-Sparks-Yau (see \cite{MSY08}) by which they showed that the derivative of the (normalized) volume function at a regular Reeb vector field is the classical Futaki invariant. Although this point will be developed in more generality in \cite{Li15b}, to the author the calculation here was an initial evidence of the validity of our theory beyond the motivations from Sasaki geometry recalled earlier. %However we will be somewhat brief here.

From now on let $V$ be a $\mathbb{Q}$-Fano variety with $\bQ$-Gorenstein klt singularities. We first recall Fujita's divisorial semistability.
\begin{defn}[{\cite[Definition 1.1]{Fuj15a}}]\label{def-Fuj}
Let $D$ be a nonzero effective Weil divisor on $V$. The pair $(V, -K_V)$ is said to be divisorial semistable along $D$ if the value:
\[
\eta(D)=\Vol_V(-K_V)-\int_0^{+\infty} \Vol_V(-K_V-xD)dx
\]
satisfies $\eta(D)\ge 0$, where $\Vol_V$ is the volume function on divisors (see \cite{Laz96}).
\end{defn}
Assume $-K_V\sim_{\mathbb{Q}} rL$ for some Cartier divisor $L$ and $r>0$. Let $X=C(V, L)$ be the affine cone over $V$ with polarization $L$. Let $\pi_0: W=Bl_oX\rightarrow X$ be the blow-up of $X$ at $o$ with the exceptional divisor still denoted by $V$. Let $D$ be a prime divisor on $V$. We think $D$ as an irreducible divisor contained in the exceptional divisor $V$ and consider the blow up $\pi_1: Y:=Bl_DW\rightarrow W$ with the exceptional divisor $E_1$. Let $\hV$ or $E_0$ denote the strict transform of $V$. Then $\pi:=\pi_1\circ\pi_0: (Y, \hV+E_1)\rightarrow (X, o)$ is a birational morphism and there are two divisorial valuations assocated to $\hV$ and $E_1$ respectively. We will compare the normalized volume of these two valuations.
For $v_{0}:=\ord_{\hV}=\ord_{V}$, it's easy to see that 
\begin{equation}\label{volv0}
\vol(v_{0})=L^{n-1}.
\end{equation}
For $v_{1}:=\ord_{E_1}$, by similar calculations as in \cite[(11)]{Kur03} and \cite{Fuj15a} (see also \cite{Li15b}), we get (compare \eqref{volvalpha}) % and \eqref{volbdot}):
\begin{lem}
\begin{equation}\label{volvinfinity}
\vol(v_1)=L^{n-1}-n\int_{0}^{+\infty}{\rm Vol}(L-xD)\frac{ dx}{(1+x)^{n+1}}.
\end{equation}
\end{lem}
To find the log discrepancy, we notice that:
\[
K_W=\pi_0^* K_X+(r-1) V, \quad K_Y=\pi_1^* K_W+E_1, \quad \pi_1^* V=\hV+E_1.
\]
Recall that $E_0=\hV$ is the strict transform of $V$ under $\pi_1$. The second identity is because $D$ is codimension $2$ inside $W$. 
So we get:
\begin{equation}\label{discsing}
K_W=\pi^* K_X+(r-1)\hV+r E_1.
\end{equation}
So we get $A_X(V)=r$ and $A_X(E_1)=r+1$.
\begin{prop}\label{thmsing}
In the above setting, the following statements hold.
\begin{enumerate}
\item
If $\ord_{V}$ minimizes $\hvol$, then $V$ is divisorial semistable along any prime divisor $D$.
\item
If $V$ is divisorial semistable along a prime divisor $D$, then $\hvol(\ord_{E_1})\ge \hvol(\ord_{V})=L^{n-1}$, where $E_1$ is the exceptional divisor over $o\in X=C(V, L)$ associated to $D$.
\end{enumerate}
\end{prop}
\begin{proof}
To compare the two normalized volume, we will first calculate the normalized volume function for quasi-monomial valuations on $(Y, \hV+E_1)$. By section \ref{secval}, for any vector $(\alpha, 1)$ with $\alpha>0$, we have a valuation $w_\alpha$ satisfying:
\begin{enumerate}
\item
$w_\alpha$ can be defined by the following condition:
\[
w_\alpha(f)\ge \alpha\cdot i+ (i+j), \mbox{ for any } f\in H^0(V, iL-jD).
\]
The identity holds if and only if $f\not\in H^0(V, iL-(j+1)D)$. % (see \eqref{rees} for more general situations)
\item $w_\alpha$ interpolates $v_0$ and $v_1$ in the normalized sense: 
\[
\lim_{\alpha\rightarrow +\infty}\frac{w_\alpha}{\alpha}=v_0, \quad \lim_{\alpha\rightarrow 0}w_\alpha=v_1.
\]
\end{enumerate}
From item 1 above, we know that $\fa_{\alpha\cdot i+(i+j)}(w_\alpha)\cap H^0(iL)=H^0(iL-jD)$. Notice that $w_\alpha$ are all $\bC^*$-invariant valuations. So the co-length of $\fa_{m+1}(w_\alpha)$ is equal to:
\[
\sum_{i=0}^{m/(\alpha+1)}\left(h^0(iL)-h^0(iL-(m-(\alpha+1)i)D)\right).
\]
It's easy to see that, modulo $m^{n-1}$, we have:
\begin{eqnarray*}
n!\sum_{i=0}^{m/(\alpha+1)}h^0(iL)=L^{n-1}\left(\frac{m}{\alpha+1}\right)^n.
\end{eqnarray*}
For the other sum, we calculate it following the method in \cite{Fuj15a}. Let $\tau(D)$ be the pseudo-effective threshold of $D$ with
respect to $(V, L)$. Then because $V$ (or its projective small $\mathbb{Q}$-factorial modification) is a Mori dream space (see \cite[Lemma 2.3]{Fuj15a}), by \cite[Theorem 4.2]{KKL12} (see \cite[Theorem 2.5]{Fuj15a}), there exist
\begin{enumerate}
\item 
an increasing sequence of rational 
numbers 
\[
0=\tau_0<\tau_1<\dots<\tau_N=\tau(D), 
\]
\item
a normal projective varieties $V_1, \dots, V_N$, and
\item
mutually distinct birational contraction maps $\phi_k: V\dashrightarrow V_k$
\end{enumerate}
such that the following hold:
\begin{enumerate}
\item for any $x\in [\tau_{k-1}, \tau_k)$, the map $\phi_k$ is a semiample model of $L-xD$, and 
\item if $x\in (\tau_{k-1}, \tau_k)$, then the map $\phi_k$ is the ample model of $L-xD$.
\end{enumerate}
For $0\le k\le N$, we define the thresholds $I(k)= m/(\alpha+1+\tau_k) $ such that
\[
\frac{m}{\alpha+1}=I(0)> I(1)>\dots\ge I(N)> I(N+1)=:0.
\]
If $I(k+1)\le i\le I(k)$ (resp. $I(k+1)<i<I(k)$), then $(m-(\alpha+1)i)/i\in [\tau_k, \tau_{k+1}]$ (resp. $(\tau_k, \tau_{k+1})$) so that 
$iL-(m-(\alpha+1)i)D=i\left(L-\frac{m-(\alpha+1)i}{i} D\right)$ is semiample (resp. ample) on $V_k$. Then {\it modulo $m^{n-1}$}, we have (see \cite[Section 4, 5]{Fuj15a}):% and \cite{Li15b}):
\begin{eqnarray*}
&&n!\sum_{i=1}^{m/(\alpha+1)} h^0(V, iL-(m-(\alpha+1)i)D)\\
&=&n! \sum_{k=0}^{N-1}\sum_{i=\lfloor I(k+1)\rfloor}^{\lfloor I(k)\rfloor-1} h^0(V_k, iL_k-(m-(\alpha+1)i)D_k)+O(m^{n-1})\\
&=&n! \sum_{k=0}^{N-1}\sum_{i=\lfloor I(k+1)\rfloor}^{\lfloor I(k)\rfloor-1}\chi(V_k, iL_k-(m-(\alpha+1)i)D_k) +O(m^{n-1})\\
&=&n\sum_{k=0}^{N-1} \int_{I(k+1)}^{I(k)}(s L_k-(m-(\alpha+1)s) D_k)^{n-1}ds+O(m^{n-1})\\
&=&n\sum_{k=0}^{N-1} \int_{I(k)}^{I(k+1)} \Vol_{V}(sL-(m-(\alpha+1)s)D)^{n-1}ds+O(m^{n-1})\\
&=&n\int_0^{\frac{m}{\alpha+1}}\Vol_V(sL-(m-(\alpha+1)s)D)ds+O(m^{n-1})\\
%\end{eqnarray*}
%Then by \cite[Proposition 4.1]{Fuj15a}, we have:
%\begin{eqnarray*}
%n!\sum_{i=1}^{m/(\alpha+1)}h^0(iL-(m-(\alpha+1)i)D)
%&=& n\int_0^{\frac{m}{\alpha+1}}{\rm Vol}(sL-(m-(\alpha+1)s)D)dx\\
&=&m^n n \int_0^{+\infty}{\rm Vol}_V(L-xD)\frac{dx}{(\alpha+1+x)^{n+1}}.
\end{eqnarray*}
So we get:
\begin{equation}\label{volvalpha}
\vol(w_\alpha)=\frac{L^{n-1}}{(\alpha+1)^n}-n\int_0^{+\infty}\Vol(L-xD)\frac{dx}{(\alpha+1+x)^{n+1}}.
\end{equation}
Notice that $(1+\alpha)^n\vol(w_\alpha)=\vol(w_\alpha/(1+\alpha))$  interpolates between \eqref{volv0} and \eqref{volvinfinity}.

For the log discrepancy of $w_\alpha$, by the definition \eqref{ldquasi} in Section \ref{secval} and \eqref{discsing}, we get:
\[
A(w_\alpha)=\alpha r+(r+1).
\]
So we get the normalized volume function:
\begin{eqnarray}\label{eq-Phibeta}
\hvol(w_\alpha)&=&A(w_\alpha)^n\cdot \vol(w_\alpha)\nonumber\\
&=&\frac{(\alpha r+(r+1))^n}{(\alpha+1)^n}L^{n-1}-n
\int_0^{\infty}\Vol(L-xD)\frac{(\alpha r+(r+1))^n}{(\alpha+1+x)^{n+1}}dx\nonumber\\
&=&\left(\frac{r+(r+1)\beta}{1+\beta}\right)^nL^{n-1}-n\beta\int_0^\infty
\Vol(L-xD)\frac{(r+(r+1)\beta)^n}{(1+(1+x)\beta)^{n+1}}dx\nonumber\\
&=:& \Phi(\beta),
\end{eqnarray}
where we substituted $\beta=\alpha^{-1}$ such that
$\Phi(0)=\hvol(v_{0})$ and $\Phi(+\infty)=\hvol(v_1)$.
We calculate the directional derivative $\Phi'(0)$:
\begin{eqnarray*}
\Phi'(0)&=&n r^{n-1} L^{n-1}- n r^{n}\int_0^{\infty}\Vol(L-xD) dx\\
&=& n\left( (-K_V)^{n-1}-\int_0^{\infty} \Vol(-K_V-xD) dx\right).
\end{eqnarray*}
Notice that the expression in the last bracket is exactly Fujita's invariant $\eta(D)$ in Definition \ref{def-Fuj}. So if $\eta(D)< 0$, then $\Phi'(0)< 0$ so that
$\hvol(w_\alpha)< \hvol(v_0)$ if $\beta=\alpha^{-1}$ is sufficiently small. This contradicts the assumption that $\hvol(v_0)$ is the minimum. So we get the first statement of Theorem \ref{thmsing}.

To see the second statement, we can write 
\[
\frac{(\alpha, 1)}{\alpha r+r+1}=\frac{\alpha r}{\alpha r+r+1}\frac{(1,0)}{r}+\frac{r+1}{\alpha r+r+1}\frac{(0,1)}{r+1}=:(1-t)\frac{(1,0)}{r}+t\frac{(0,1)}{r+1},
\]
where we introduced a new variable $t=\frac{r+1}{\alpha r+r+1}$, or equivalently $\alpha=\beta^{-1}=\frac{1-t}{t}\frac{r+1}{r}$. So we get
$w_\alpha/(\alpha r+r+1)=:\tilde{v}_t$ is a ``linear" interpolation between $\tilde{v}_0=v_0/r$ and $\tilde{v}_1=v_1/(r+1)$. Notice that $A(\tilde{v}_t)\equiv 1$.

Denote $f(t):=\hvol(\tilde{v}_t)=\Phi(\beta(t))$. The directional derivative of $f(t)$ at $t=0$ is calculated by the chain rule:
\[
f'(0)=\Phi'(0)\beta'(0)=n\cdot \eta(D)\cdot \frac{r}{r+1}.
\]
We claim that $f(t)$ is a convex function of $t$. Assuming the claim, we know that $f'(0)\ge 0$ implies $f(1)\ge f(0)$. So we obtain the second statement of \ref{thmsing}. To verify the claim, we re-write the formula of $\Phi(\beta)$ in \eqref{eq-Phibeta} by the integration by parts:
\begin{eqnarray*}
\frac{\Phi(\beta)}{(r+(r+1)\beta)^n}&=&\frac{L^n}{(1+\beta)^n}+\left.\frac{\Vol(L-xD)}{(1+(1+x)\beta)^n}\right|_{x=0}^{+\infty}-\int_0^{+\infty}\frac{d \Vol(L-xD)}{(1+(1+x)\beta)^n}\\
&=&-\int^{+\infty}_0 \frac{d \Vol(L-xD)}{(1+(1+x)\beta)^n}.
\end{eqnarray*}
So, with $\beta=\frac{t r}{(1-t)(r+1)}$, we easily get:
\begin{eqnarray*}
f(t)&=&\Phi(\beta(t))=r^n(r+1)^n\int^{+\infty}_0 \frac{-d \Vol(L-xD)}{(r+1+(rx-1)t)^n}.
\end{eqnarray*}
Because $\Vol(L-xD)$ is a decreasing function of $x\in [0, +\infty)$, $-d \Vol(L-xD)$ is a measure with positive density with respect to the Lebesgue measure $dx$. So the claim follows from the fact that
$t\mapsto (r+1+(rx-1)t)^{-n}$ is a convex function on $[0, 1]$.

\end{proof}

\section{Questions and discussions}

\subsection{Some conjectures}\label{sec-conj}
There are several natural questions that deserve further studies. We collect them into following conjectures which we plan to study. 
The calculations in this paper should be viewed as evidences for these conjectures.
\begin{conj}\label{mainconj}
\begin{enumerate}
\item
{\bf Hypothesis C} is true. As a consequence, for any germ of $\bQ$-Gorenstein klt singularity $(X,o)$, there exists a minimizer (denoted by $v_*=v_*(X,o)$) of $\hvol(v)$ by Corollary \ref{maincor}. 
\item
$v_*$ is unique up to positive rescaling. %and in particular the minimum $\hvol(v_*)$ is unique.
\item 
$v_*$ is always a quasi-monomial valuation.
\item 
Let $V$ be a Fano manifold, and let $X=C(V, -K_V)$. Then $V$ is K-semistable if and only if on the cone singularity $(X,o)$, $\hvol(v)$ is minimized at $\ord_V$ where $V$ is considered as the exceptional divisor of the blow up $\pi: Bl_oX\rightarrow X$. 
\item
Assume $(X,o)$ is a $\bQ$-Gorenstein klt singularity on a K\"{a}hler-Einstein Fano variety $(X, \omega_{\rm KE})$. The minimizer $v_*(X,o)$ for $(X,o)$ is exactly the weight function that gives the metric tangent cone at $o\in (X, d_{\omega_{\rm KE}})$. More precisesly, we consider the the associated graded algebra of $v_*$:
\[
{\rm gr}_{v_*}R=\bigoplus_{m\in \Phi} \fa_m(v_*)/\fa_{>m}(v_*),
\]
where $\Phi$ is the valuation semigroup of $v_*$. Then the conjecture is that ${\rm gr}_{v_*}R$ is finitely generated and normal, and ${\rm Spec}\left({\rm gr}_{v_*}R\right)$ specially degenerates to the metric tangent cone at $(X, o)$. If true, this is an answer to a question %first asked by Tian to the author in 2012 (\cite{Tia2}), and first appeared in the literature 
of Donaldson-Sun \cite{DS15}. 

\item
For (Newton non-degenerate) hypersurface klt singularities in $\bC^{n+1}$, the global minimizers of $\hvol$ can be found among the valuations induced by weighted blow ups of the ambient $\bC^{n+1}$.
%As special examples, we conjecture that the weights found in Example \ref{exhyper} are indeed global minimizers of $\hvol(v)$.
\end{enumerate}
\end{conj}

{\bf Postscript Note:} 
After the initial writing of this paper, there have been progresses on the above conjectures (see \cite{Li15b, LL16, Blu16}). In particular Blum proved the existence of minimizers without verifying the lower semicontinuity of $\hvol$ (but using the main estimates in this paper). However the quasi-monomial part and uniqueness part are still open in general.
%The 4th part of the above conjecture has been proved in \cite{Li15b} and \cite{LL16}. Results on the existence and uniqueness of minimizers will appear in \cite{Li15b} and \cite{LX16}. 
%In particular, we prove in \cite{LL16} and \cite{LX16} that the valuations obtained for $A_k (k\ge 1)$ and $E_k (k=6,7,8)$ in the Example 2.8 are indeed global minimizers of $\hvol$.
%\section{Appendix A: Izumi's linear complementary inequality}
%In this appendix, we briefly recall the proof for the following version of Izumi's linear complementary inequalities. 

\section{Appendix I: Candidate minimizer for D-type and E-type singularities}\label{app-DE}

\begin{enumerate}

\item
Consider the $(n+1)$-dimensional $D_{k+1}$ singularity:
\[
X^{n+1}=D^{n+1}_{k+1}:=\{z_1^2+\dots+z_n^2+z_{n+1}^2 z_{n+2}+z_{n+2}^k=0\}\subset\mathbb{C}^{n+2}.
\]
with $n\ge 1$ and $k\ge 3$.

Consider the valuation determined by $\bx=(x_1, \dots, x_n, x_{n+1}, x_{n+2})$ with $x_1\dots x_n$. Then
\[
A_X(v_{\bx})=\sum_{i=1}^{n+2}x_i-\min\{2x_1, 2x_{n+1}+x_{n+2}, k x_{n+2}\}, \vol(v_{\bx})=\frac{\min\{2x_1, 2x_{n+1}+x_{n+2}, kx_{n+2}\}}{\prod_{i=1}^{n+2} x_i}.
\]
The minimization of $\hvol(v_\bx)$ is a standard multivariable calculus problem. % (see Figure \ref{figDk} for illustration).
Although the complete discussion is messy, %and we put it in the Appendix \ref{secDk}, 
the end results are clean and are orgainized in the following table.

\vskip 2mm
\begin{center}
\renewcommand*{\arraystretch}{1.5}
\begin{tabular}{|ctc|c|c|c}
\hline
%\theadfont
\diagbox[width=3em]{$k$}{${\scriptstyle n+1}$}
%&\thead{Second\\column}&\thead{Third\\column}
& $2$ & $3$ & $4$ & $\ge 5$ \\
\thickhline
$3$ & $(1,\frac{2}{3},\frac{2}{3})$     &    $(1,1,\frac{2}{3},\frac{2}{3})$      &  $(1, 1, 1, \frac{2}{3},\frac{2}{3})$  & $(1,\dots, 1, \frac{n-2}{n-1}, \frac{n-2}{n-1})$ \\
$k\ge 4$& $(1,\frac{k-1}{k},\frac{2}{k})$     &    $(1,1,\alpha_*(2),2-2\alpha_*(2))$      &  $(1,1,1,\alpha_*(3),2-2\alpha_*(3))$  & $(1, \dots, 1, \frac{n-2}{n-1}, \frac{n-2}{n-1})$ \\
\hline
\end{tabular}
\end{center}
\vskip 2mm
Here $\alpha_*(n)=\frac{-n+\sqrt{5n^2-4n}}{2(n-1)}$: $\alpha_*(2)\approx 0.732$, $\alpha_*(3)\approx 0.686$. Denoting the weights in the above table by $w_*$, we have the following cases of associated degenerations:
\begin{enumerate}
\item $k=3$ and $2\le n+1\le 5$, or $n+1=2$ and any $k$, the weight 
in the above table comes from a natural $\bC^*$-action preserving $X$. The corresponding normalized volume $\hvol(w_*)=\frac{((n-1)k+1)^{n+1}}{k^{n-1}(k-1)}$.
\item
For $n+1=3, 4$ and $k\ge 4$ (irrational $w_*$ with $\hvol(w_*)=\phi_2(\alpha_*)$), or $n+1=5$ and $k\ge 4$ ($w_*=(1, \dots, 1, \frac{2}{3}, \frac{2}{3})$ with $\hvol(w_*)=\frac{(3n-2)^{n+1}}{2\cdot 3^{n-1}}$), the corresponding weight ``degenerates" $X$ to the (non-isolated) singularity given by
\[
Y^{n+1}=\{z_1^2+\cdots+z_n^2+z^2_{n+1}z_{n+2}=0\}\subset \bC^{n+2}.
\]
\item
For any other case ($n+1\ge 6$ and $k\ge 3$), the weight $w_*=(1, \dots, 1, \frac{n-2}{n-1}, \frac{n-2}{n-1})$, with $\hvol(w_*)=2(n+1)^{n+1}\frac{(n-2)^{n-1}}{(n-1)^{n-1}}$, degenerates $X$ to the singularity:
\[
A^{n-1}_1\times \bC^2=\{z_1^2+\cdots+z_n^2=0\}\subset \bC^{n+2}.
\]
\end{enumerate}

\item
Consider the $(n+1)$-dimensional $E_7$ singularity:
\[
X^{n+1}=E^{n+1}_7:=\{z_1^2+z_2^2+\dots+z_n^2+z_{n+1}^3z_{n+2}+z_{n+2}^3=0\}\subset\bC^{n+2}.
\]
By renaming the variables, we can rearrange the weight $\bx$ such that $x_1\le x_2\le\dots\le x_n$. So the problem is to minimize the functional:
\[
\hvol(v_\bx)=\left(\sum_{i=1}^{n+2}x_i-\min\{2x_1, 3x_{n+1}+x_{n+2}, 3x_{n+2}\}\right)^{n+1}\frac{\min\{2x_1, 3x_{n+1}+x_{n+2}, 3x_{n+2}\}}{\prod_{i=1}^{n+2}x_i}
\]

We get:
\begin{enumerate}
\item $n\ge 5$. The unique minimizer is the weight $w_*=(1, \dots, 1, \frac{n-2}{n-1}, \frac{n-2}{n-1})$ with $\hvol(w_*)=2(n+1)^{n+1}\frac{(n-2)^{n-1}}{(n-1)^{n-1}}$, degenerating $X$ to 
$A^{n-1}_1\times \bC^2=\{z_1^2+\cdots+z_n^2=0\}\subset \bC^{n+2}$. Notice that $A^{n-1}_1\times \bC^2$ has a product Ricci-flat K\"{a}hler cone metric as explained above.
\item $n=4$. The unique minimizer is the weight $w_*=(1, 1, 1, 1, \frac{2}{3}, \frac{2}{3})$ with $\hvol(w_*)=\frac{50000}{27}$, degenerating $X$ to 
$A^{4}_{2}\times \bC^1=\{z_1^2+\cdots+z_4^2+z_6^3=0\}\subset \bC^6$ . Notice that $A^{4}_2\times \bC^1$ is semistable in the sense as explained above.
\item $n=3$. The unique minimizer is the weight $w_*=(1, 1, 1, \frac{5}{9}, \frac{2}{3})$,with $\hvol(w_*)=\frac{32000}{243}$, degenerating $X$ to $A^{3}_2\times \bC^1=\{z_1^2+z_2^2+z_3^2+z_5^3=0\}\subset\bC^5$. Notice that $A^3_2\times \bC^1$ has a product Ricci-flat K\"{a}hler cone metric as explained above.
\item $n=2$. The unique minimizer is the weight $w_*=(1,1,\frac{4}{9}, \frac{2}{3})$ with $\hvol(w_*)=\frac{250}{27}$, coming from the natural $\bC^*$-action.
\item $n=1$. The unique minimizer is the weight $(1,\frac{4}{9}, \frac{2}{3})$ with $\hvol(w_*)=\frac{1}{12}$, coming form the natural $\bC^*$-action.
\end{enumerate}

\item 
Consider the $(n+1)$-dimensional $E_6$ singularity:
\[
X^{n+1}=E^{n+1}_6:=\{z_1^2+\dots+z_n^2+z_{n+1}^3+z_{n+2}^4=0\}\subset\bC^{n+2}.
\]
Consider the weights $\bx$ such that $x_1\le\dots\le x_n$. The results are: %By Appendix \ref{secE6}, we get:
\begin{enumerate}
\item
$n\ge 5$, the unique minimizer is $w_*=(1, \dots, 1, \frac{n-2}{n-1}, \frac{n-2}{n-1})$, degenerating $X$ to $A^{n-1}_1\times \bC^2$.
\item
$n=4$, the unique minimizer is $(1, 1, 1, 1, \frac{2}{3}, \frac{2}{3})$, degenerating $X$ to 
$A^4_2\times \bC^1=\{z_1^2+\cdots+z_4^2+z_5^3=0\}\subset \bC^6$.
\item
$n=3$, the unique minimizer is $(1, 1, 1, \frac{2}{3}, \frac{5}{9})$, degenerating $X$ to 
$A^3_2\times \bC^1=\{z_1^2+z_2^2+z_3^2+z_4^3=0\}\subset \bC^5$.
\item
$n=2$, the unique minimizer is $(1,1,\frac{2}{3}, \frac{1}{2})$ with $\hvol=\frac{343}{36}$, coming from the natural $\bC^*$-action.
\item
$n=1$, the unique minimizer is $(1, \frac{2}{3}, \frac{1}{2})$ with $\hvol=\frac{1}{6}$, coming from the natural $\bC^*$-action.

\end{enumerate}

\item
Consider the $(n+1)$-dimensional $E_8$ singularity:
\[
X^{n+1}=E^{n+1}_8:=\{z_1^2+z_2^2+\dots+z_n^2+z_{n+1}^3+z_{n+2}^5=0\}\subset \mathbb{C}^{n+2}.
\]
Consider the valuation determined by $\bx=(x_1,\dots, x_{n}, \alpha, \beta)$. Then
\[
A_X(v_{\bx})=\sum_{i=1}^n x_i+\alpha+\beta-\min\{2, 3\alpha, 5\beta\}, \quad \vol(v_{\bf x})=\frac{\min\{2x_1, 3\alpha, 5\beta \}}{\left(\prod_{i=1}^n x_i \right)\alpha\beta}.
\]
\begin{enumerate}
\item $n\ge 5$. $w_*=(1,\dots,1,\frac{n-2}{n-1}, \frac{n-2}{n-1})$, degenerating $X$ to $A^{n-1}_1\times\bC^2\subset \bC^{n+2}$.
\item $n=4$. $w_*=(1, 1, 1, 1, 2/3, 2/3)$, degenerating $X$ to $A^{4}_2\times \bC^1$.
\item $n=3$. $w_*=(1,1,1,2/3,5/9)$, degenerating $X$ to $A^3_2\times\bC^1$.
\item $n=2$. $w_*=(1,1,2/3,2/5)$ with $\hvol=\frac{2048}{225}$, coming from the natural $\bC^*$-action.
\item $n=1$. $w_*=(1,2/3,2/5)$ with $\hvol=\frac{1}{30}$, coming from the natural $\bC^*$-action.
\end{enumerate}

\end{enumerate}

\begin{rem}
In the A-D-E type examples, notice that if $\dim X=2$ then $\hvol(w_*)=\frac{4}{|G|}$, where $G$ is given by, $\bZ_{k}$ for $A^2_{k-1}$, binary dihedral group of order $4(k-1)$ for $D_{k+1}$, and
binary tetrahedral, octahedral, icosahedral groups for $E^2_6, E^2_7, E^2_8$ respectively. See \cite{LL16} for a general result for quotient singularities.
\end{rem}

\section{Appendix II: Second proof of Theorem \ref{thm-UIzumi} }\label{app-izumi}

In this appendix, we present a direct proof of Theorem \ref{thm-UIzumi} pointed out to me by an anonymous referee. This proof is more in the spiritual of \cite{Izu85} and is based on the argument of Boucksom-Favre-Jonsson in \cite[Section 3.1]{BFJ12} %(see also \cite[Lemma 9.4]{BBEGZ}) 
and the following  

\noindent
{\bf 
Fact: } (see \cite{Betal09}) For any smooth projective variety $X'$ and any ample line bundle $L\rightarrow X'$, there exists a positive constant 
$\epsilon>0$ such that for any $x\in X'$, the divisor $\pi^*L-\epsilon E$ is ample where $\pi$ is the blow-up at $x$ and $E$ is the exceptional divisor of $\pi$. 

First, by compactifying $X'$, we can assume $X'$ is projective and $L$ is a very ample line bundle over $X'$. Moreover we can assume $X'$ is smooth and $\mu^{-1}(o)$ is a connected simple normal crossing divisor (not necessarily reduced) whose reduced support is given by $\sum_{i=1}^m F_i$. Indeed, the connectedness follows from Zariski's main theorem. Moreover we can take a log resolution of $(X', \mu^{-1}(o))$ and the uniform estimate on the log resolution is easily seen to imply the uniform estimate for $(X', \mu^{-1}(o))$. 
For any $o'\in \mu^{-1}(o)$, let $\pi_{o'}: Bl_{o'}X'\rightarrow X'$ be the blow-up of $o'$ with the exceptional divisor denoted by $E_0$. By the above fact, we can choose $\epsilon$ sufficiently small so that $M=M_{o'}=\pi_{o'}^*L-\epsilon E_0$ is ample for any $o'\in \mu^{-1}(o)$. 
\begin{rem}
Although we don't need this, under the log-smoothness assumption, the dual complex $\Delta$ of $\mu^{-1}(o)$ is connected and the dual complex $\Delta_{o'}$ of $(\pi_{o'}\circ\mu)^{-1}(o)$ is obtained by either attaching a new segment at a vertex of $\Delta$ or taking a barycenter subdivision of a face of $\Delta$. 
\end{rem}
We will denote by $E_i$ the strict transform of $F_i$ under the blow up $\pi_{o'}$. Assume $g\in \mathcal{O}_{X, o}$ and let $G=\{\mu^*g=0\}$ be the effective divisor. We write
\[
G=b_0 E_0+\sum_{i=1}^m b_i E_i+\tilde{G}=\sum_{i=0}^{m}b_i E_i + \tilde{G}
\]
where $\tilde{G}$ is an effective divisor whose support does not contain any $E_i$. Notice that we have
\begin{equation}
b_i=\ord_{F_i}(\mu^*g)=\ord_{F_i}(g), \text{ for } 0\le i\le m.
\end{equation} 
In particular, we have $b_0=\ord_{E_0}(g)=\ord_{o'}(\mu^* g)$. Next consider the intersection:
\begin{eqnarray}\label{eq-intineq}
\sum_{j=0}^m b_j (E_i\cdot E_j\cdot M^{n-2})&=&(G\cdot E_i\cdot M^{n-2})-(\tilde{G}\cdot E_i\cdot M^{n-2})\nonumber\\
&\le& (G\cdot E_i\cdot M^{n-2})=0.
\end{eqnarray}
The last identity is because $G$ is a principle divisor near $E_i$. Set 
\[
c_{i,j}:=(E_i\cdot E_j\cdot M^{n-2}).
\]
Then from \eqref{eq-intineq} we get the inequality:
\[
\sum_{j\neq i} b_j c_{i,j}\le - b_i c_{i,i}\le b_i |c_{i,i}|.
\]
Without the loss of generality, we assume $E_0\cap E_1\neq \emptyset$. Then $c_{i,j}\ge 0$ if $j\neq i$ and 
\begin{align*}
&c_{1,1}=(E_1\cdot E_1\cdot (\pi_{o'}^*L-\epsilon E_0)^{n-2})=(F_1\cdot F_1\cdot L^{n-2})-\epsilon^{n-2},\\
&c_{0, 1}=(E_0\cdot E_1\cdot (\pi_{o'}^*L-\epsilon E_0)^{n-2})=\epsilon^{n-2}>0.
\end{align*}
So we get $b_0 \le \frac{|c_{1,1}|}{c_{0,1}} b_1$ which is equivalent to:
\begin{equation}
\ord_{o'}(\mu^* g)\le \frac{|c_{1,1}|}{c_{0,1}} \ord_{F_1}(g).
\end{equation}
By the original Izumi's theorem, we know that for each $i$ there exists a constant $d_i>0$ such that
\[
\ord_{F_i}(\mu^* g)\le d_i\cdot \ord_o(g).
\] 
So if we choose
\[
a_2=\max\left\{d_i \frac{\left|(F_i\cdot F_i\cdot L^{n-2})-\epsilon^{n-2}\right|}{\epsilon^{n-2}}, 1\le i\le m\right\},
\]
then we have the desired inequality:
\[
\ord_{o'}(\mu^*g)\le a_2\cdot \ord_o(g), \text{ for any } g\in \mathcal{O}_{X,o}.
\]

\section{Acknowledgment}
The author is partially supported by NSF DMS-1405936. I would like to thank Yuchen Liu, Xiaowei Wang and Chenyang Xu for helpful discussions, and Gang Tian and Bernd Ulrich for their interest. I am grateful to Sebastien Boucksom and Mattias Jonsson for pointing out to me the existence of valuations with infinite log discrepancies. I also thank M. Jonsson for helping me realize the issue in Remark \ref{remcont}. 
%and bringing the reference \cite{Hu14} to my attention. 
I also would like to thank Yuji Odaka for helpful comment regarding the example 2.8. I am grateful to an anonymous referee for providing constructive suggestions on the presentation of paper and pointing out the second proof of Theorem \ref{thm-UIzumi} in Appendix II.

\vskip 1mm
\noindent
Department of Mathematics, Purdue University, West Lafayette, IN 47907-2067 USA

\noindent
li2285@purdue.edu

\end{document}